\documentclass[final]{siamltex}
\usepackage{amssymb}
\usepackage{graphicx,fancyhdr}
\usepackage{multirow,amsmath}
\usepackage{dsfont}
\usepackage{subfigure}
\usepackage{bbm}
\usepackage{booktabs}
\usepackage{bbm}%数学符号字体的设置
\usepackage{amsfonts}
\usepackage{mathrsfs,float}
\usepackage{color,bm}
\usepackage{algorithm}
\usepackage{algpseudocode}
\makeatletter
\newtheorem{thm}{Theorem}[section]

\newtheorem{remark}[thm]{Remark}
\usepackage[pagewise]{lineno}

\title{  Fast numerical derivatives  based on multi-interval Fourier extension\thanks{The research is partially supported by National Natural Science Foundation of China (Nos. 12261131494, 12171455, RSF-NSFC 23-41-00002)}}

\author{Zhenyu Zhao\thanks{School of Mathematics and Statistics, Shandong University of Technology, Zibo, 255049, China,({\tt Zhenyu\_Zhao@sdut.edu.cn}).} \and Yanfei Wang \thanks{Key Laboratory of Deep Petroleum Intelligent Exploration and Development, Institute of Geology and Geophysics, Chinese Academy of Sciences, Beijing, 100029,China,({\tt yfwang@mail.iggcas.ac.cn}).}\and Xinran Liu \thanks{School of Mathematics and Statistics, Shandong University of Technology, Zibo, 255049, China.}}

\begin{document}

\maketitle

\begin{abstract}
We present a computationally efficient algorithm for stable numerical differentiation from noisy, uniformly-sampled data
on a bounded interval. The method combines multi-interval Fourier extension approximations with an adaptive domain
partitioning strategy: a global precomputation of local Fourier sampling matrices and their thin SVDs is reused
throughout a recursive bisection procedure that selects locally-resolved Fourier fits. Each accepted subinterval stores
a compact set of Fourier coefficients that are subsequently used to reconstruct the derivative via a precomputed
differentiation operator. The stopping criterion balances fitting error and an explicit noise-level bound, and the
algorithm automatically refines the partition where the function exhibits rapid oscillations or boundary activity.
Numerical experiments demonstrate significant improvements over existing methods, achieving accurate derivative reconstruction for challenging functions. The approach provides a robust framework for ill-posed differentiation problems while maintaining computational efficiency.
\end{abstract}

\begin{keywords}
Numerical Differentiation, Fourier Extension, Fourier Continuation, Adaptive Partitioning, Regularization
\end{keywords}

\begin{AMS}
42A10, 65T40, 65T50
\end{AMS}

\pagestyle{myheadings}
\thispagestyle{plain}
\markboth{Z. Y. ZHAO and Y.F. Wang and X. R. Liu}{ Fast numerical derivatives  based on multi-interval Fourier extension}
\section{Introduction}
Numerical differentiation, the process of approximating derivatives from discrete and often noisy function values, constitutes a fundamental yet inherently ill-posed problem in scientific computing. Its significance stems from pervasive applications across diverse fields,  such as  heat flux estimation\cite{bazan2019}, volatility modeling in finance \cite{2005Tikhonov,2006Recovering}, solving some differential and integral equations \cite{1991Abel,2023Accurate,egidi2021},  and discontinuity reconstruction and edge detection\cite{2002A,2006Detection}. The core challenge lies in the intrinsic instability: minute perturbations or noise in the input function data inevitably amplify, leading to large, unbounded errors in the computed derivatives. This ill-posedness necessitates the application of specialized regularization techniques to recover stable and meaningful approximations.

A wide spectrum of computational strategies has been developed to address this challenge.
Traditional approaches include finite difference schemes and polynomial interpolation \cite{qu1996,li2005,sergeyev2011},
which, while simple, exhibit pronounced sensitivity to noise and grid non-uniformity.
Spline-based methods \cite{wei2006,hao2012,dellaccio2024,moniri2023} offer greater robustness by leveraging piecewise polynomial smoothing,
though controlling the inherent smoothness regularization remains critical.
Mollification techniques \cite{murio1993,friedrichs1944} stabilize differentiation through convolution with smooth kernels,
effectively damping high-frequency noise.
Methods utilizing reproducing kernel Hilbert spaces (RKHS) or radial basis functions (RBF) \cite{davydov2016,wei2007,davydov2018,wei2006}
provide flexibility for scattered data approximation but depend heavily on kernel selection and regularization parameter tuning.
Tikhonov regularization, in its standard and modified forms \cite{zhao2021,chen2018,nayak2020,egidi2021,bozkurt2021},
remains a dominant paradigm, penalizing undesirable solution norms within a least-squares minimization of the data misfit.
Spectral methods, particularly those employing truncated Fourier or Legendre series \cite{bazan2019,zhao2010,zhao2010b,semenova2022,zhang2024,zhao2021}, achieve high accuracy for smooth periodic functions, benefiting from convergence rates that adapt to the underlying smoothness. Integration-based approaches \cite{huang2014,wang2010,wang2011} exploit the inherent stability of the inverse operation (integration), while partial differential equation (PDE)-based methods \cite{qiu2018,wang2015} reformulate differentiation as an inverse source problem, leveraging PDE theory for regularization, albeit often at the cost of increased complexity.

Among these, Fourier-based spectral methods are particularly attractive for their potential high-order convergence on periodic smooth data. To overcome their limitation to periodic functions, Fourier extension methods have emerged as a powerful technique \cite{huybrechs2010fourier,adcock2019frames,zhao2024fast,zhao2025new,zhao2025local}. The core idea involves extending the target function $f$, defined on a limited interval $[-1, 1]$ and belonging to a Sobolev space $H^p(-1, 1)$, to a function $f_e$ that is periodic on a larger interval $[-T, T](T>1)$ , such that $f_e(x) = f(x)$ on $[-1, 1]$. This extension mitigates the Gibbs phenomenon by effectively transferring the non-periodicity boundary issues to the larger domain. Building upon this foundation, one of the authors of this paper and his collaborators  introduced a significant advancement by incorporating a novel \emph{super-order regularization} strategy within the Tikhonov framework \cite{chen2018}. Their method seeks an approximation $g = \mathscr{F}\mathbf{v}$ (where $\mathscr{F}$ maps the coefficient vector $\mathbf{v}$ to a trigonometric series) to the noisy data $f^\delta$ on $[-1, 1]$ by minimizing the functional:
\begin{equation}
\Phi(\mathbf{v}) = \| \mathscr{F}\mathbf{v} - f^{\delta} \|_{L^2(-1,1)}^2 + \alpha  \| \mathscr{R}\mathbf{v} \|_{l^2}^2.
\end{equation}
The key innovation lies in the regularization operator $\mathscr{R}$, defined as
$$\mathscr{R}\mathbf{v} = (c_0, e^{\frac{\pi}{2}}c_1, e^{\frac{\pi}{2}}s_1, \ldots, e^{\frac{n\pi}{2}}c_n, e^{\frac{n\pi}{2}}s_n, \ldots)^T,$$
which corresponds to penalizing the $L^2(-2,2)$ norm of $\sum_{l=0}^{\infty} D^l g / l!$. This super-order penalty imposes an exponentially weighted constraint on \emph{all} derivatives simultaneously. Suppose that the noisy data satisfy the condition
$$\|f-f^{\delta}\|_{L^2(-1,1)}\leq \delta,$$
then the regularization parameter $\alpha$ is selected adaptively using the discrepancy principle ($\| \mathscr{F}\mathbf{v}_{\alpha}^{\delta} - f^{\delta} \|_{L^2(-1,1)} = C\delta,  C>1$). Theoretically, this approach achieves order-optimal convergence rates $$\| D^k (f^{\delta} - f) \|_{L^2(-1,1)} = \mathcal{O}(\delta^{\frac{p-k}{p}})$$ for $k \leq p$ \emph{without requiring explicit knowledge of the smoothness index $p$}.

While the Fourier extension method with super-order regularization demonstrates strong potential--particularly in effectively addressing non-periodic functions and achieving adaptive convergence---several limitations remain. Both numerical experiments and theoretical analysis indicate the need for further improvement in the following aspects:

\begin{itemize}
\item {\bf Difficulty with high-frequency functions:} The regularization penalty suppresses high-frequency components, and floating-point precision imposes additional constraints. As a result, the method fails to accurately approximate derivatives of high-frequency functions such as $\cos(100x^2)$, even when data noise is negligible.

\item {\bf Inhibited accuracy for low-frequency functions:} Solving the minimization problem requires forming the corresponding normal equation. According to [2], the best attainable approximation accuracy is $\mathcal{O}(\sqrt{\epsilon})$, where $\epsilon$ is machine precision. This error is further magnified during differentiation, making it difficult to achieve higher accuracy even for smooth, low-frequency functions.

\item {\bf Lack of fast algorithms:} To date, this class of methods lacks efficient implementations, which limits its practical applicability due to high computational costs.
\end{itemize}

To overcome these challenges, the authors recently proposed a local Fourier extension algorithm \cite{zhao2025local}. This approach avoids solving normal equations by truncating the singular values of the operator $\mathscr{F}$ directly, thereby improving the approximation accuracy to $\mathcal{O}(\epsilon)$. Furthermore, it leverages the rapid convergence of Fourier approximations to achieve high accuracy with few nodes for low-frequency functions. For high-frequency functions, interval partitioning effectively reduces local frequency, and the reuse of a unified discrete matrix across subdomains enhances computational efficiency. In this paper, we apply this method to numerical differentiation, addressing the three key limitations outlined above.

The remainder of this paper is organized as follows. In Section~\ref{sec:method} we introduce the multi-interval Fourier
extension formulation of the numerical differentiation problem and present the main regularization ideas that underpin our approach.
Section~\ref{sec:implementation} is devoted to the numerical implementation: we describe the uniform sampling discretization, the precomputation of the local Fourier bases and their singular value decompositions, and the adaptive domain partitioning strategy used to produce stable local fits. In Section~\ref{sec:tests} we summarize a comprehensive set of numerical experiments that
illustrate the method's accuracy and robustness to noise for a variety of smooth and highly oscillatory test functions.
Finally, Section~\ref{sec:conclusion} collects concluding remarks and possible directions for future work.
\section{Numerical Differentiation Method Based on Multi-Interval Fourier Extension\label{sec:method}}

\subsection{Single-Interval Generalized Truncated Singular Value Method}

Let $\Omega = [0, 2\pi]$, $\Lambda = \left[0, \frac{2\pi}{T} \right]$ with $T > 1$. Define  the system
\begin{equation}
\Phi = \{\phi_\ell\}_{\ell \in \mathbb{Z}}, \quad \phi_0(t) = \frac{1}{\sqrt{2}},\quad \phi_\ell(t) = e^{\mathrm{i}\ell t},\quad \ell = \pm1, \pm2, \ldots.
\end{equation}
As shown in \cite{huybrechs2010fourier}, this system forms an orthonormal basis for $L^2(\Omega)$ and a frame for $L^2(\Lambda)$.

Define the following operators:
\begin{equation}
\begin{aligned}
 &\mathcal{F}: \ell^2 \rightarrow L^2(\Lambda), \quad {c} = \{\hat{c}_\ell\}_{\ell=-\infty}^\infty \mapsto \sum_{\ell=-\infty}^{\infty} \hat{c}_\ell \phi_\ell, \\
 &\mathcal{P}_N: \ell^2 \rightarrow \ell^2, \quad {c} \mapsto (\ldots, 0, \hat{c}_{-N}, \ldots, \hat{c}_N, 0, \ldots)^\top, \\
 &B: \mathcal{D}(B) \subset \ell^2 \rightarrow \ell^2, \quad {c} \mapsto \left\{ \sqrt{1 + |\ell|^{2p}}\, \hat{c}_\ell \right\}_{\ell=-\infty}^{\infty}, \\
 &W: \mathcal{D}(W) \subset \ell^2 \rightarrow \ell^2, \quad {c} \mapsto \left\{ e^{|\ell|} \hat{c}_\ell \right\}_{\ell=-\infty}^{\infty}.
\end{aligned}
\end{equation}

Consider the operator equation
\begin{equation}\label{compactoperatoreq}
\mathcal{F} {c} = g,
\end{equation}
and its perturbed form
\begin{equation}\label{compactoperatoreqpert}
\mathcal{F} {c} = g^\delta,
\end{equation}
with the noise assumption
\begin{equation}\label{errorcond1}
\|g - g^\delta\|_{L^2(\Lambda)} \leq \delta.
\end{equation}

By the Sobolev extension theorem, for $g \in H^p(\Lambda)$, equation \eqref{compactoperatoreq} admits multiple solutions. Among them, we define:

- The **least-squares solution** ${c}^\star \in \ell^2$ satisfies
\[
\|\mathcal{F} {c}^\star - g\|_{L^2(\Lambda)} = \inf_{{c} \in \ell^2} \|\mathcal{F} {c} - g\|_{L^2(\Lambda)}.
\]

- The **$B$-weighted best approximation** ${c}_B^\dag \in \ell^2$ satisfies
\[
\|{c}_B^\dag\|_{\ell^2} = \inf\left\{ \|B {c}^\star\|_{\ell^2} \,\middle|\, {c}^\star \text{ is a least-squares solution of \eqref{compactoperatoreq}} \right\}.
\]

Let $G = \mathcal{F} W^{-1}$, which is clearly a compact operator. Let $(\sigma_i, v_i, u_i)$ be a singular system of $G$ with
\[
\sigma_1 \geq \sigma_2 \geq \cdots \geq \sigma_k \geq \cdots > 0.
\]
For any $g \in L^2(\Lambda)$, define the truncated inverse operator:
\begin{equation}
\mathcal{T}_k g = \sum_{i=1}^{k} \frac{1}{\sigma_i} \langle g, u_i \rangle v_i.
\end{equation}
It holds that
\begin{equation}\label{Tnorm}
\|\mathcal{T}_k\| \leq \frac{1}{\sigma_k}.
\end{equation}

We further have the following estimate:
\begin{lemma}
\label{lem:G-estimates}
The following bounds hold:
\begin{equation}\label{operatorest}
\|I - G \mathcal{T}_k\| \leq 1, \quad \|(I - G \mathcal{T}_{k-1}) G\| \leq \sigma_k.
\end{equation}
\end{lemma}

\begin{proof}
For any \( g \in L^2(\Lambda) \), write
\[
g = \sum_{i=1}^\infty \langle g, u_i \rangle u_i.
\]
Then we can get
\[
(I - G \mathcal{T}_{k-1}) g = \sum_{i=k}^\infty \langle g, u_i \rangle u_i,
\]
and
\[
\| (I - G \mathcal{T}_{k-1}) g \|^2 = \sum_{i=k}^\infty |\langle g, u_i \rangle|^2 \leq \|g\|^2.
\]
Therefore,
\[
\| I - G \mathcal{T}_{k-1} \| \leq 1.
\]

Moreover, for any unite vector \( {c}\in \mathcal{D}(G) \), write
\[
{c}= \sum_{i=1}^\infty \langle {c}, v_i \rangle v_i, \quad \text{so that} \quad G c= \sum_{i=1}^\infty \sigma_i \langle c, v_i \rangle u_i.
\]
It can be deduced that
\[
(I - G \mathcal{T}_{k-1}) G c = \sum_{i=k}^\infty \sigma_i \langle c, v_i \rangle u_i,
\]
and hence,
\[
\| (I - G \mathcal{T}_{k-1}) G c \|^2 = \sum_{i=k}^\infty \sigma_i^2 |\langle c, v_i \rangle|^2
\leq \sigma_k^2 \sum_{i=k}^\infty |\langle c, v_i \rangle|^2 \leq \sigma_k^2 \|c\|^2.
\]
We conclude that
\[
\| (I - G \mathcal{T}_{k-1}) G \| \leq \sigma_k.
\]

\end{proof}

We define the generalized truncated SVD (GTSVD) solution to \eqref{compactoperatoreqpert} as
\begin{equation}\label{TSVDsol}
{c}^{\delta, \eta} = W^{-1} \mathcal{T}_{\eta} g^\delta,
\end{equation}
where the truncation parameter $\eta$ is chosen according to the discrepancy principle:
\begin{equation}\label{discreprin}
\|\mathcal{F} {c}^{\delta, \eta} - g^\delta\|_{L^2(\Lambda)} \leq \tau \delta < \|\mathcal{F} {c}^{\delta, \eta - 1} - g^\delta\|_{L^2(\Lambda)}, \quad \tau > 1.
\end{equation}
Then we set
\begin{equation}\label{approxtog}
\mathbf{g}_{\eta}^\delta=\mathcal{Q}_{\eta}^\delta g^{\delta}: = \mathcal{F}{c}^{\delta, \eta}
\end{equation}
as the approximation to $g$.

We also define
\begin{equation}\label{defmid}
{c}_{B, N} = \mathcal{P}_N {c}_B^\dag, \quad {c}^{\eta} = W^{-1} \mathcal{T}_{\eta} g, \quad {c}_N^{\eta} = W^{-1} \mathcal{T}_{\eta} \mathcal{F} {c}_{B,N}.
\end{equation}
Since $g \in H^p(\Lambda)$, there exists a constant $E > 0$ such that
\[
\|B {c}_B^\dag\|_{\ell^2} \leq E.
\]
And we have the following result:
\begin{lemma}
The approximants ${c}_B^\dag$ and ${c}_{B,N}$ satisfy
\begin{equation}\label{projest}
\|{c}_B^\dag - {c}_{B,N}\|_{\ell^2} \leq N^{-p} E, \quad \|W {c}_{B,N}\|_{\ell^2} \leq C_N E,
\end{equation}
where
\begin{equation}\label{paraCN}
C_N = \max\left(1, \frac{e^N}{N^p}\right).
\end{equation}
\end{lemma}
\begin{proof}First,
  \begin{equation*}
    \begin{aligned}
      \|{ c}^{\dag}_B-{ c}_{B,N}\|_{\ell^2}&=&&\sum_{|\ell|>N}\hat{c}_{\ell}^2\leq \frac{1}{N^{2p}}\sum_{|\ell|>N}(1+{\ell}^{2p})\hat{c}_{\ell}^2\leq  \frac{1}{N^{2p}}\|B{c}^{\dag}_B\|_{\ell^2}^2.
    \end{aligned}
  \end{equation*}
  Moreover,
\begin{equation*}
  \begin{aligned}
     \|W{ c}_{B,N}\|_{\ell^2}=\sum_{\ell=-N}^Ne^{\ell}\hat{c}_{\ell}^2\leq C_N^2\sum_{\ell=-N}^N(1+{\ell}^{2p})\hat{c}_{\ell}^2\leq C_N^2E^2.
  \end{aligned}
\end{equation*}
\end{proof}
\begin{lemma}[\cite{chen2018}]\label{lemsuperorder}
If the vector sequence $\{{c}^\delta\}$ satisfies
\begin{equation}
\|\mathcal{F} {c}^\delta\|_{L^2(\Lambda)} \leq k_1 \delta, \quad \|W {c}^\delta\|_{\ell^2} \leq k_2 e^{k_3 \delta^{-1/p}} \delta, \quad \text{as } \delta \to 0,
\end{equation}
for some nonnegative constants $k_1, k_2, k_3$, then there exists a constant $E_0$ such that
\begin{equation}
\|\mathcal{F} (B {c}^\delta)\|_{L^2(\Lambda)} \leq E_0.
\end{equation}
\end{lemma}

\begin{lemma}[\cite{adams1975sobolev}]\label{interpineq}
Let $\Theta = (a, b) \subset \mathbb{R}$ and suppose $f \in H^s(\Theta)$. Then for any $\epsilon \in (0, \epsilon_0]$ and $0 \leq j \leq p$, there exists a constant $C$ (depending on $\epsilon_0$, $j$, and $p$) such that
\begin{equation}
\|f^{(j)}\|_{L^2(\Theta)} \leq C \left( \epsilon \|f^{(p)}\|_{L^2(\Theta)} + \epsilon^{-j/(p-j)} \|f\|_{L^2(\Theta)} \right).
\end{equation}
\end{lemma}

\begin{theorem}
Suppose $g \in H^p(\Lambda)$ and let $\mathbf{g}_{\eta}^\delta$ be defined as in \eqref{approxtog}, with $\delta$ satisfying \eqref{errorcond1}. Then for any $0 \leq j \leq p$, the following convergence rate holds:
\begin{equation}\label{finalerrorest}
\left\| \left(\mathbf{g}_{\eta}^\delta \right)^{(j)} - g^{(j)} \right\|_{L^2(\Lambda)} = \mathcal{O} \left( \delta^{\frac{p - j}{p}} \right).
\end{equation}
\end{theorem}

\begin{proof}
Using \eqref{errorcond1}, \eqref{discreprin}, \eqref{projest}, and the triangle inequality, we obtain
\begin{equation}\label{est1}
\begin{aligned}
\|\mathcal{F} ({c}^{\delta,\eta} - {c}_{B,N})\|_{L^2(\Lambda)}
&\leq \|\mathcal{F} {c}^{\delta,\eta} - g^\delta\|_{L^2(\Lambda)} + \|g^\delta - g\|_{L^2(\Lambda)} + \|g - \mathcal{F} {c}_{B,N}\|_{L^2(\Lambda)} \\
&\leq (\tau + 1)\delta + \| {c}_B^\dag - {c}_{B,N} \|_{\ell^2} \leq (\tau + 1)\delta + N^{-p}E.
\end{aligned}
\end{equation}
Using \eqref{Tnorm}, \eqref{operatorest}, and \eqref{projest}, we estimate
\begin{equation}\label{est2}
\begin{aligned}
\|S({c}^{\delta,\eta} - {c}_{B,N})\|_{\ell^2}
&\leq \|W({c}^{\delta,\eta} - {c}^{\eta}_N)\|_{\ell^2} + \|W({c}^{\eta}_N - {c}_{B,N})\|_{\ell^2} \\
&\leq \frac{1}{\sigma_{\eta}} \left( \|g^\delta - g\|_{L^2(\Lambda)} + \|g - \mathcal{F} {c}_{B,N}\|_{L^2(\Lambda)} \right) + C_N E \\
&\leq \frac{1}{\sigma_{\eta}} (\delta + N^{-p}E) + C_N E.
\end{aligned}
\end{equation}
Let $\mathcal{U} = I - G \mathcal{T}_{\eta - 1}$. Note that
\[
g^\delta - \mathcal{F} \mathbf{c}^{\delta, \eta - 1} = \mathcal{U} g^\delta,
\]
so
\begin{equation*}
\begin{aligned}
\tau \delta &< \|g^\delta - \mathcal{F} \mathbf{c}^{\delta, \eta - 1}\|_{L^2(\Lambda)} \\
&\leq \|\mathcal{U} (g^\delta - g)\|_{L^2(\Lambda)} + \|\mathcal{U}(g - \mathcal{F} \mathbf{c}_{B,N})\|_{L^2(\Lambda)} + \|\mathcal{U} \mathcal{F} \mathbf{c}_{B,N}\|_{L^2(\Lambda)} \\
&\leq \delta + N^{-p}E + \sigma_{\eta} C_N E.
\end{aligned}
\end{equation*}
Choose $N$ such that
\begin{equation}\label{est3}
N^{-p} E = \frac{\tau - 1}{2} \delta,
\end{equation}
then
\begin{equation}\label{est4}
\frac{\delta}{\sigma_{\eta}} \leq \frac{2 C_N E}{\tau - 1}.
\end{equation}
Combining \eqref{est1}--\eqref{est4} and using the bound \eqref{paraCN}, we conclude that there exist constants $k_1, k_2, k_3$ such that
\begin{equation}
\begin{aligned}
\|\mathcal{F}({c}^{\delta,\eta} - {c}_{B,N})\|_{L^2(\Lambda)} &\leq k_1 \delta, \\
\|W({c}^{\delta,\eta} - {c}_{B,N})\|_{\ell^2} &\leq k_2 e^{k_3 \delta^{-1/p}} \delta.
\end{aligned}
\end{equation}
Then, by Lemma~\ref{lemsuperorder}, there exists a constant $E_0$ such that
\begin{equation}
\|\mathcal{F} B({c}^{\delta,\eta} - {c}_{B,N})\|_{L^2(\Lambda)} \leq E_0.
\end{equation}
Therefore,
\begin{equation}\label{finalest1}
\begin{aligned}
\|g^{(p)} - (\mathbf{g}^{\delta}_{\eta})^{(p)}\|_{L^2(\Lambda)}
&\leq \|\mathcal{F} B ({c}^{\delta,\eta} - {c}_B^\dag)\|_{L^2(\Lambda)} \\
&\leq \|\mathcal{F} B ({c}^{\delta,\eta} - {c}_{B,N})\|_{L^2(\Lambda)} + \|B {c}_B^\dag\|_{\ell^2} \\
&\leq E_0 + E.
\end{aligned}
\end{equation}
Moreover,
\begin{equation}\label{finalest2}
\|\mathbf{g}_{\eta}^\delta - g\|_{L^2(\Lambda)} \leq \|\mathbf{g}_{\eta}^\delta - g^\delta\|_{L^2(\Lambda)} + \|g^\delta - g\|_{L^2(\Lambda)} \leq (\tau + 1)\delta.
\end{equation}
Finally, applying the interpolation inequality (Lemma~\ref{interpineq}) completes the proof of the theorem.
\end{proof}
\begin{remark}
  According to the regularization theory, the result in \eqref{finalerrorest} is of optimal order, and the design of the method is independent of the smoothness $p$ and is adaptive. However, we need to note that the above method has serious numerical defects. Limited by the machine precision (set as $\epsilon$), we cannot obtain the exact value of the singular value of $\sigma_i< \epsilon$. According to the results in \cite{zhao2025new}, even if $g$ is accurate, we can only obtain the approximate accuracy of
  \begin{equation*}
    \begin{aligned}
     \left\|g-\sum_{\sigma_i\geq\epsilon}\langle g, u_i\rangle \right\|&=&&   \left\|\sum_{\sigma_i<\epsilon}\langle g, u_i\rangle \right\|=   \left\|\sum_{\sigma_i< \epsilon}\langle \mathcal{F}{ c}_B^{\dag}, u_i\rangle \right\|= \left\|\sum_{\sigma_i< \epsilon}\langle{G}(W{c}_B^{\dag}), u_i\rangle \right\|\\
     &=&&\left\|\sum_{\sigma_i< \epsilon}\sigma_i\langle W{c}_B^{\dag}, v_i\rangle \right\|\leq \epsilon\|W{c}_B^{\dag}\|_{\ell^2}.\\
    \end{aligned}
  \end{equation*}
For high-frequency functions, the norm $\|W{ c}_B^{\dag}\|_{\ell^2}$ is very large. We know that $e^{40}\approx 2.35e+17$, while $\epsilon=1e-16$ in general double-precision calculations. That is to say, when the function g contains high-frequency components that cannot be ignored, the above method cannot obtain satisfactory results even from the perspective of approximation, not to mention unstable calculations such as derivatives.
\end{remark}
\subsection{Multi-interval processing}

Let $f\in H^p(a,b)$ be a  non-periodic function defined on the interval $I=[a,b]$ and $f^{\delta}$ be its perturbation measurement data.  Points $a=a_0<a_1<\cdots <a_J=b$ divides $I$ into $J$ subintervals $I_1,I_2,\ldots, I_J$. Let $f_j(x), f_j^{\delta}(x)$ be functions defined on interval $I_j$, satisfying:
\begin{equation}
  f_j(x)=f(x),\quad f^{\delta}_j(x)=f^{\delta}(x),\quad \forall x\in I_j,\quad j=1,2,\ldots,J,
\end{equation}
and
\begin{equation}\label{errorconf}
  \|f_j-f^{\delta}_j\|_{L^2(I_j)}\leq \delta_j.
\end{equation}
Our goal is to obtain an effective approximation  of $f$ through the perturbed  data $f^{\delta}$.

Let $\{g_k(t)\}_{j=1}^J$ and $\{g^{\delta}_k(t)\}_{j=1}^J$ are two families of functions defined on $\Lambda$, such that
\begin{equation}
  g_j(t)=f_j(a_{j-1}+d_j t),\quad  g^{\delta}_j(t)=f^{\delta}_j(a_{j-1}+d_j t),\quad \forall t\in \Lambda,
\end{equation}
where
\begin{equation}
  d_j=\frac{T}{2\pi}(a_j-a_{j-1}).
\end{equation}
Then we have
\begin{equation}
  \|g_j- g^{\delta}_j\|_{L^2(\Lambda)}\leq \sqrt{d_j}\delta_j=:\tilde{\delta}_j.
\end{equation}
Now substitute $g^{\delta}_j$, $j=1,2,...J$ into the right side of equation \eqref{compactoperatoreqpert},   we can get a set of solution vectors from \eqref{TSVDsol} and \eqref{discreprin}:
\begin{equation}
  {c}^{\tilde{\delta}_j,\eta_j},\quad j=1,2,\ldots,J.
\end{equation}
and the corresponding approximations of the function $g_j$:
\begin{equation}
  \mathbf{g}_{\eta_j}^{\tilde{\delta}_j}=\mathcal{Q}_{\eta_j}^{\tilde{\delta}_j}g^{\delta}_j=\mathcal{F}{ c}^{\tilde{\delta}_j,\eta_j},\quad j=1,2,\ldots,J.
\end{equation}
Now we use
\begin{equation}\label{approximationto_f}
  {\bf f}^{\bm \delta}=\left(\mathcal{P}_{\bm \eta,J}^{\bm{\delta}}f^{\delta}\right)(x)=
  \left\{\begin{array}
    {lllll}
    \left(\mathcal{Q}{\eta_j}^{\tilde{\delta}_j}g^{\delta}_j\right)\left(\frac{x-a_j}{d_j}\right),& \forall x\in [a_{j-1},a_k),\\
    \left(\mathcal{Q}{\eta_1}^{\tilde{\delta}_1}g^{\delta}_1\right)\left(2\pi\right), &x=b.
  \end{array}\right.
\end{equation}
as the approximation of $f$ and we have the following estimation.
\begin{theorem}
  Suppose that $f\in H^p(a,b)$ and the condition \eqref{errorconf} holds, ${\bf f}^{\bm \delta}$ is defined by \eqref{approximationto_f}, and let
  \begin{equation}
    \underline{d}=\min_{1\leq j\leq J} d_j,\quad {\bar{\delta}=\max_{1\leq j\leq J}\tilde{\delta}_j},
  \end{equation}
  then we have
  \begin{equation}\label{finalerrorest2}
    \left\|\left({\bf f}^{\bm \delta}\right)^{(k)}-f^{(k)}\right\|_{L^2(a,b)}={\mathcal{O}\left(\underline{d}^{-k}\bar{\delta}^{\frac{p-k}{p}}\right).}
  \end{equation}
\end{theorem}
\begin{proof}
\begin{equation}
  \begin{aligned}
    \left\|\left({\bf f}^{\bm \delta}\right)^{(k)}-f^{(k)}\right\|^2_{L^2(a,b)}&=&&\int_a^b\left[\left({\bf f}^{\bm \delta}\right)^{(k)}(x)-f^{(k)}(x)\right]^2dx\\
    &=&&\sum_{j=1}^J\int_{I_j}\left[\left(\mathcal{Q}{\eta_j}^{\tilde{\delta}_j}g^{\delta}_j\right)^{(k)}\left(\frac{x-a_j}{d_j}\right)-g^{(k)}\left(\frac{x-a_j}{d_j}\right)\right]^2dx\\
    &=&&\sum_{j=1}^J\int_{\Lambda}\frac{1}{d_j^{2k-1}}\left[\left(\mathcal{Q}{\eta_j}^{\tilde{\delta}_j}g^{\delta}_j\right)^{(k)}\left(t\right)-g^{(k)}\left(t\right)\right]^2dt\\
    &=&&\frac{T}{2\pi}\sum_{j=1}^J\frac{a_j-a_{j-1}}{d_j^{2k}}\left\|\left(\mathcal{Q}{\eta_j}^{\tilde{\delta}_j}g^{\delta}_j\right)^{(k)}-g^{(k)}\right\|_{L^2(\Lambda)}^2\\
    &\leq&&\frac{T(b-a)}{2\pi}\sum_{j=1}^J\frac{1}{d_j^{2k}}\left\|\left(\mathcal{Q}{\eta_j}^{\tilde{\delta}_j}g^{\delta}_j\right)^{(k)}-g^{(k)}\right\|_{L^2(\Lambda)}^2\\
  \end{aligned}
\end{equation}
\end{proof}
\begin{remark}
  \begin{itemize}
    \item  Comparing the theoretical results of \eqref{finalerrorest} and \eqref{finalest1}, it seems that there is no substantial advantage of the multi-interval method, but we need to note that the function $f(x)=e(\text{i}\pi\omega x)$ corresponds to
  \begin{equation}\label{omegachange}
 g_j(t)=e^{\text{i}\omega\pi a_{j-1}}e^{\text{i}\omega\pi s_j t}.
  \end{equation}
That is to say, for high-frequency functions, we can make each $g_j$ a low-frequency function by dividing the region, and its corresponding Fourier coefficients converge quickly, and the corresponding $\|W{ c}_B^{\dag}\|_{\ell^2}$ can be maintained within the range allowed by the machine precision, so that the function $g_j$ can be effectively approximated at a smaller discrete scale.
\item Later, we will show in the numerical implementation part that we can use the same discrete matrix to solve in each subinterval $I_j$, thereby greatly reducing the computational complexity of the problem.
\item The results of \eqref{finalest1} also show that over-division of the interval $[a,b]$ will reduce the accuracy. The scale of division should be related to the frequency of the function f. How to determine the scale of division, we will give a posterior method combining errors in the numerical implementation part.
  \end{itemize}
\end{remark}
\section{Numerical implement based on uniform sampling \label{sec:implementation}}
\subsection{Discretization of equation \eqref{compactoperatoreqpert}}
From the previous analysis, we can see that we need to repeatedly solve equation \eqref{compactoperatoreqpert} with different right-hand terms. For fixed $n$, let
$${\Phi}_n=\{\phi_{\ell}\}_{|\ell|\leq n}.$$
Then take the sampling ratio $\gamma\geq 1$ and let
$$m=\lceil \gamma(2n+1)\rceil,$$
where $\lceil \cdot\rceil$ is the rounding symbol. When we select parameters $n,\gamma, T$ and take
 $$L=\lceil T\times m\rceil, \quad h=\frac{2\pi}{T},\quad t_i=ih,\quad i=0,1,\ldots, m-1.$$
  Then we can obtain the discrete form of equation \eqref{compactoperatoreqpert}:
  \begin{equation}
    \label{discretform}
    \mathbf{F}_{\gamma,n}^{T}{\bf c}=\mathbf{g}^{\delta},
  \end{equation}
  where
  \begin{equation*}
     \left( \mathbf{F}_{\gamma,n}^{T}{\bf c}\right)_{i,\ell}=\frac{1}{\sqrt{L}}\phi_{\ell}(t_i),\quad \left(\mathbf{g}^{\delta}\right)_i=\frac{1}{\sqrt{m}}g^{\delta}(t_i).
  \end{equation*}
    \begin{remark}
   The parameters $n, \gamma$ and $T$ have a substantial impact on the properties of the discrete matrix, which in turn affects the final calculation results. Reference \cite{zhao2025local} conducted a large number of empirical tests on the relevant parameters, which we will not repeat here. In the following, we use the following parameters for calculation:
   $$n=9, \gamma=1, T=6.$$
Take $$\mathbf{G}_{\gamma,n}^{T}=\mathbf{F}_{\gamma,n}^{T}{\mathbf{W}^{-1}_n},$$
   where $ {\mathbf W}_n$ is  a diagonal matrix with
   \begin{equation*}
({\bf d})_{\ell}=e^{|\ell|},\quad \ell=-n,\ldots,n.
   \end{equation*}
  as the diagonal vector. Now we can obtain the SVD of matrix $\mathbf{G}_{\gamma,n}^{T}$ as
  \begin{equation}
    \mathbf{G}_{\gamma,n}^{T}= \sum_{j=1}^{2n+1}\mathbf{u}_j{\bm\sigma}_j\mathbf{v}_j^T,
  \end{equation}
  where
  ${\bm\sigma}_j$ are the singular values of $ \mathbf{G}_{\gamma,n}^{T}$ such that
  \begin{equation*}
    {\bm \sigma}_1\geq{\bm \sigma}_2\geq\ldots {\bm\sigma}_{2n+1}\geq 0,
  \end{equation*}
  while ${\bf u}_i$ and ${\bf v}_i$ are the left and right singular vectors of $ \mathbf{G}_{\gamma,n}^{T}$, respectively.

  Now if we assume that the exact data ${\bf g}$ and the perturbation data ${\bf g}^{\delta}$ satisfy:
  \begin{equation}
    \|{\bf g}^{\delta}-{\bf g}\|_{\ell^2}\leq {\bm\delta},
  \end{equation}
  Then the generalized truncated singular value decomposition solution of \eqref{discretform} can be given as:
  \begin{equation}\label{discretsol}
   {\bf c}^{{\bm \delta},{\bm\eta}}={\bf W}_n^{-1} \left(\sum_{i=1}^{{\bm\eta}}\frac{1}{{\bm\sigma}_i}\langle {\bf g}^{\delta},{\bf u}_i\rangle {\bf v}_i\right).
  \end{equation}
  where ${\bm\eta}$ is determined by the discrepancy principle
  \begin{equation}
   \left\|\mathbf{F}_{\gamma,n}^{T}{\bf c}^{{\bm \delta},{\bm\eta}}-{\bf g}^{\delta}\right\|_{\ell^2}\leq {\bm\tau}{\bm \delta}<
   \left\|\mathbf{F}_{\gamma,n}^{T}{\bf c}^{{\bm \delta},{\bm\eta}-1}-{\bf g}^{\delta}\right\|_{\ell^2}
  \end{equation}
  with ${\bm \tau} >1$.

  \end{remark}
  \subsection{Adaptive domain partitioning strategy}

The adaptive domain partitioning strategy forms the computational backbone of our multi-interval approach. The algorithm
performs local Fourier-extension fits on adaptively refined subintervals and accepts a fit when its reconstruction error,
appropriately measured and weighted by the assumed noise level, falls below a relaxed tolerance. To allow unlimited bisection, the number of global sampling nodes is chosen as
\[
M = 2^{r}(m-1)+1,
\]
where $m$ is the minimal number of nodes required in a single interval and $r$ is a nonnegative integer chosen so that
the global grid aligns with dyadic subdivisions; after downsampling, local fits always use $m$ reference samples so that
linear algebra quantities are consistent with the same $m$-point local fitting rule.

Next, we present a detailed description of the multi-interval (recursive bisection) implementation and give pseudocode
that corresponds precisely to the MATLAB implementation. The overall procedure consists of the following four main stages:
\begin{enumerate}
\item \textbf{Precomputation}: Given $n$, $m$, $L$, $T$, compute:
\begin{itemize}
\item Sub-Fourier matrix $\mathbf{G} \in \mathbb{C}^{m \times (2n+1)}$
\item Extended matrix $\mathbf{TG} \in \mathbb{C}^{L \times (2n+1)}$ for high-resolution reconstruction
\item Derivative matrices $\mathbf{dG} = \mathbf{G} \cdot \operatorname{diag}(i k \cdot 2\pi/T)$, $\mathbf{dTG} = \mathbf{TG} \cdot \operatorname{diag}(i k \cdot 2\pi/T)$
\item SVD of $\mathbf{G} = \mathbf{U} \mathbf{S} \mathbf{V}^T$
\end{itemize}
These are computed once and reused throughout.

\item \textbf{Local Fourier Extension Fitting} (Algorithm~\ref{alg:LocalFourierFit}):
For subinterval $[a_{\text{sub}}, b_{\text{sub}}]$ with noisy data $\mathbf{y}_d \in \mathbb{R}^{nn}$:
\begin{itemize}
\item Compute downsampling ratio $r_{\text{ds}} = (nn-1)/(m-1)$
\item Extract reduced vector $\mathbf{y}_{d,\text{red}} = \mathbf{y}_d[1:r_{\text{ds}}:nn]$
\item Solve via truncated SVD: $\mathbf{Coef} = \text{LocalFourierFit}(\mathbf{y}_{d,\text{red}}, \delta_{\text{loc}}, m, \mathbf{U}, \mathbf{S}, \mathbf{V}, \mathbf{G})$
\item Reconstruct: $\mathbf{FE} = \Re(\mathbf{TG} \cdot \mathbf{Coef})$
\item Interpolate to original resolution: $\mathbf{refinedFE} = \operatorname{interpft}(\mathbf{FE}, r_{\text{ds}} \cdot L)$
\end{itemize}

\item \textbf{A Posteriori Error Criterion}:
Compute residual
 $$\epsilon_{\text{res}} = \|\mathbf{y}_d - \mathbf{refinedFE}(1:nn)\|_2.$$
 Given noise bound $\delta_1$, define $\epsilon_{\text{target}} = \sqrt{nn/3} \delta_1$.
  If $$\epsilon_{\text{res}} \leq \rho \epsilon_{\text{target}} (\rho=2) \quad \text{or}\quad  nn \leq m,$$
   accept the subinterval; otherwise, bisect and recurse.

\item \textbf{Derivative Reconstruction} (Algorithm~\ref{alg:MultiDiffInterp}):
For each accepted subinterval with coefficients $\mathbf{Coef}$:
\begin{itemize}
\item Compute derivative approximation: $\mathbf{dFE} = \mathbf{dTG} \cdot \mathbf{Coef} \cdot Scale$ where $Scale = 1/(b_{\text{sub}} - a_{\text{sub}})$
\item Interpolate: $\mathbf{dFE}_{\text{interp}} = \operatorname{interpft}(\mathbf{dFE}, r_{\text{ds}} \cdot L)$
\item Concatenate results across subintervals
\end{itemize}
\end{enumerate}

\begin{remark}{Implementation Notes}
\begin{itemize}
\item \textbf{Global SVD reuse}: The SVD of $\mathbf{G}$ is computed once, enabling fast projections without refactorization.
\item \textbf{Consistent downsampling}: $M = 2^r(m-1)+1$ ensures $r_{\text{ds}}$ is integer-valued at all recursion levels.
\item \textbf{Termination condition}: Recursion stops when $nn \leq m$ to prevent undersampling.
\item \textbf{Noise adaptation}: The local tolerance $\delta_{\text{loc}} = \delta_1 \sqrt{m/3}$ assumes uniform noise variance $\delta_1^2/3$.
\end{itemize}
\end{remark}

\begin{algorithm}
\caption{Local Fourier Extension Fitting via Truncated SVD \label{alg:LocalFourierFit}}
\begin{algorithmic}[1]
\Function{LocalFourierFit}{$\mathbf{y}_d, \delta_{\text{loc}}, m, \mathbf{U}, \mathbf{S}, \mathbf{V}, \mathbf{G}$}
\State $\mathbf{Coef} \gets \mathbf{0}_{m \times 1}$, $\mathbf{ry} \gets \mathbf{0}_{m \times 1}$, $k \gets 1$
\State $\epsilon_{\text{res}} \gets \|\mathbf{y}_d\|_2$
\While{$\epsilon_{\text{res}} > \delta_{\text{loc}}$ and $k \leq m$}
    \State $\mathbf{c}_k \gets \mathbf{V}(:,k) \cdot (\mathbf{U}(:,k)^T \mathbf{y}_d / \mathbf{S}(k,k))$
    \State $\mathbf{Coef} \gets \mathbf{Coef} + \mathbf{c}_k$
    \State $\mathbf{ry} \gets \mathbf{ry} + \mathbf{G} \mathbf{c}_k$
    \State $\epsilon_{\text{res}} \gets \|\mathbf{ry} - \mathbf{y}_d\|_2$
    \State $k \gets k + 1$
\EndWhile
\State \Return $\mathbf{ry}, \mathbf{Coef}$
\EndFunction
\end{algorithmic}
\end{algorithm}

\begin{algorithm}
\caption{Recursive Multi-Interval Fitting \label{alg:RecursiveInterp}}
\begin{algorithmic}[1]
\Function{RecursiveInterp}{$a, b, \mathbf{y}_d, m, \delta_1, \mathbf{U}, \mathbf{S}, \mathbf{V}, \mathbf{G}, \mathbf{TG}$}
\State $nn \gets \text{length}(\mathbf{y}_d)$
\State $r_{\text{ds}} \gets (nn-1)/(m-1)$
\State $\mathbf{y}_{d,\text{red}} \gets \mathbf{y}_d[1:r_{\text{ds}}:nn]$
\State $\delta_{\text{loc}} \gets \delta_1 \sqrt{m/3}$
\State $(\mathbf{ry}, \mathbf{Coef}) \gets$ \Call{LocalFourierFit}{$\mathbf{y}_{d,\text{red}}, \delta_{\text{loc}}, m, \mathbf{U}, \mathbf{S}, \mathbf{V}, \mathbf{G}$}
\State $\mathbf{FE} \gets \Re(\mathbf{TG} \cdot \mathbf{Coef})$
\State $\mathbf{refinedFE} \gets \operatorname{interpft}(\mathbf{FE}, r_{\text{ds}} \cdot L)$
\State $\epsilon_{\text{res}} \gets \|\mathbf{y}_d - \mathbf{refinedFE}(1:nn)\|_2$
\State $\epsilon_{\text{target}} \gets \sqrt{nn/3} \delta_1$
\If{$\epsilon_{\text{res}} \leq \rho \epsilon_{\text{target}}$ or $nn \leq m$}
    \State \Return $[a; b; \mathbf{Coef}]$
\Else
    \State $mid \gets (nn-1)/2 + 1$ \Comment{Ensure integer index}
    \State {\small$\mathbf{coeffs}_L \gets$ \Call{RecursiveInterp}{$a, \frac{a+b}{2}, \mathbf{y}_d[1:mid], m, \delta_1, \mathbf{U}, \mathbf{S}, \mathbf{V}, \mathbf{G}, \mathbf{TG}$}}
    \State {\small$\mathbf{coeffs}_R \gets$ \Call{RecursiveInterp}{$\frac{a+b}{2}, b, \mathbf{y}_d[mid:nn],m,\delta_1, \mathbf{U}, \mathbf{S}, \mathbf{V}, \mathbf{G}, \mathbf{TG}$}}
    \State \Return $[\mathbf{coeffs}_L, \mathbf{coeffs}_R]$
\EndIf
\EndFunction
\end{algorithmic}
\end{algorithm}

\begin{algorithm}
\caption{Derivative Reconstruction \label{alg:MultiDiffInterp}}
\begin{algorithmic}[1]
\Function{MultiDiffInterp}{$a, b, M, m, L, \mathbf{coeffs}, \mathbf{dTG}$}
\State $\mathbf{diffy} \gets [\,]$, $L_{\text{global}} \gets b - a$
\For{each subinterval record $(a_j, b_j, \mathbf{Coef})$ in $\mathbf{coeffs}$}
    \State $K \gets L_{\text{global}} / (b_j - a_j)$
    \State $r_{\text{ds}} \gets (M-1)/(m-1) / K$
    \State $\mathbf{dFE} \gets \mathbf{dTG} \cdot \mathbf{Coef} \cdot K / (b - a)$
    \State $\mathbf{dFE}_{\text{interp}} \gets \operatorname{interpft}(\mathbf{dFE}, r_{\text{ds}} \cdot L)$
    \State Append $\mathbf{dFE}_{\text{interp}}[1:r_{\text{ds}} \cdot (m-1)]$ to $\mathbf{diffy}$
\EndFor
\State Append last point of $\mathbf{dFE}_{\text{interp}}$ to $\mathbf{diffy}$
\State \Return $\mathbf{diffy}$
\EndFunction
\end{algorithmic}
\end{algorithm}

\section{Numerical tests\label{sec:tests}}
In this section, we assess the performance of the proposed algorithm through a series of numerical experiments.
All computations were performed on a Windows~10 system with 16~GB of memory and an Intel(R) Core(TM)~i7-8500U CPU @~1.80~GHz, using \textsc{MATLAB}~2016b.
We compare the algorithm (M1) introduced herein with the full-data algorithm (M2) of~\cite{chen2018}.
For M2, we fix the parameters $T = 2$ and $\gamma = 2$ to facilitate implementation.
Unless otherwise specified, the maximum recursion depth is set to $r = 6$, corresponding to
\[
M = 2^6 \times (19 - 1) + 1 = 1153
\]
sampling nodes for the reconstruction of derivatives from perturbed data.
For lower-frequency functions, smaller values of $r$ may be used in practice to reduce the computational cost.

The relative derivative error is defined by
\[
  RE = \frac{\left( \frac{1}{M} \sum_{i=1}^M \left[ f'(x_i) - \left( {\bf f}^{\bm\delta} \right)' (x_i) \right]^2 \right)^{1/2}}
  {\left( \frac{1}{M} \sum_{i=1}^M \left[ f'(x_i) \right]^2 \right)^{1/2}}.
\]

We consider the following three classes of test functions:
\begin{itemize}
\item \emph{Low-frequency functions}:
\[
  f_1(x) = e^x, \quad f_2(x) = x^3 - 3x^2 + \frac{x}{2}.
\]
\item \emph{High-frequency functions within the interval}:
\[
  f_3(x) = \cos\!\left( \frac{100}{1 + 25x^2} \right), \quad f_4(x) = \mathrm{erf}(x).
\]
\item \emph{High-frequency functions near the interval boundaries}:
\[
  f_5(x) = \cos(100x^2), \quad f_6(x) = \frac{1}{1.1 - x^2}.
\]
\end{itemize}

\begin{figure}
			\begin{center}
\subfigure[\label{1a} $f(x)=\exp(x)$] {
\resizebox*{6cm}{!}{\includegraphics{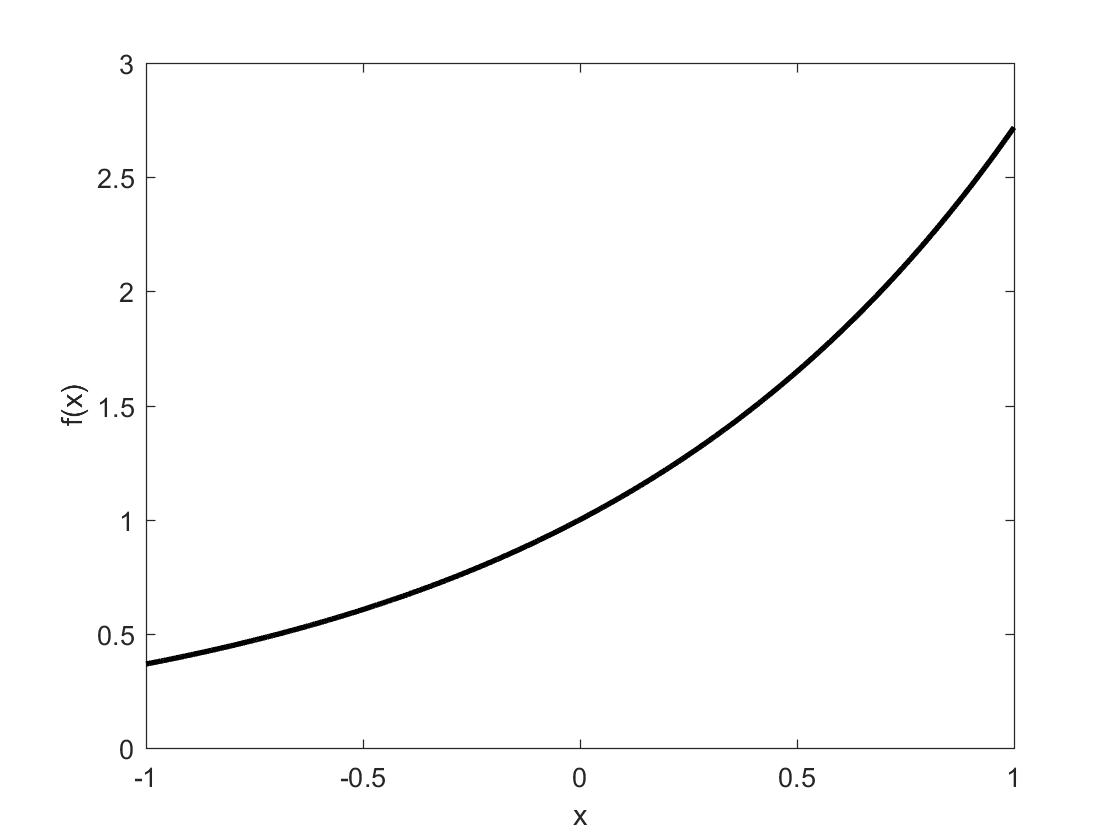}}
}%
\subfigure[\label{1b} relative error] {
\resizebox*{6cm}{!}{\includegraphics{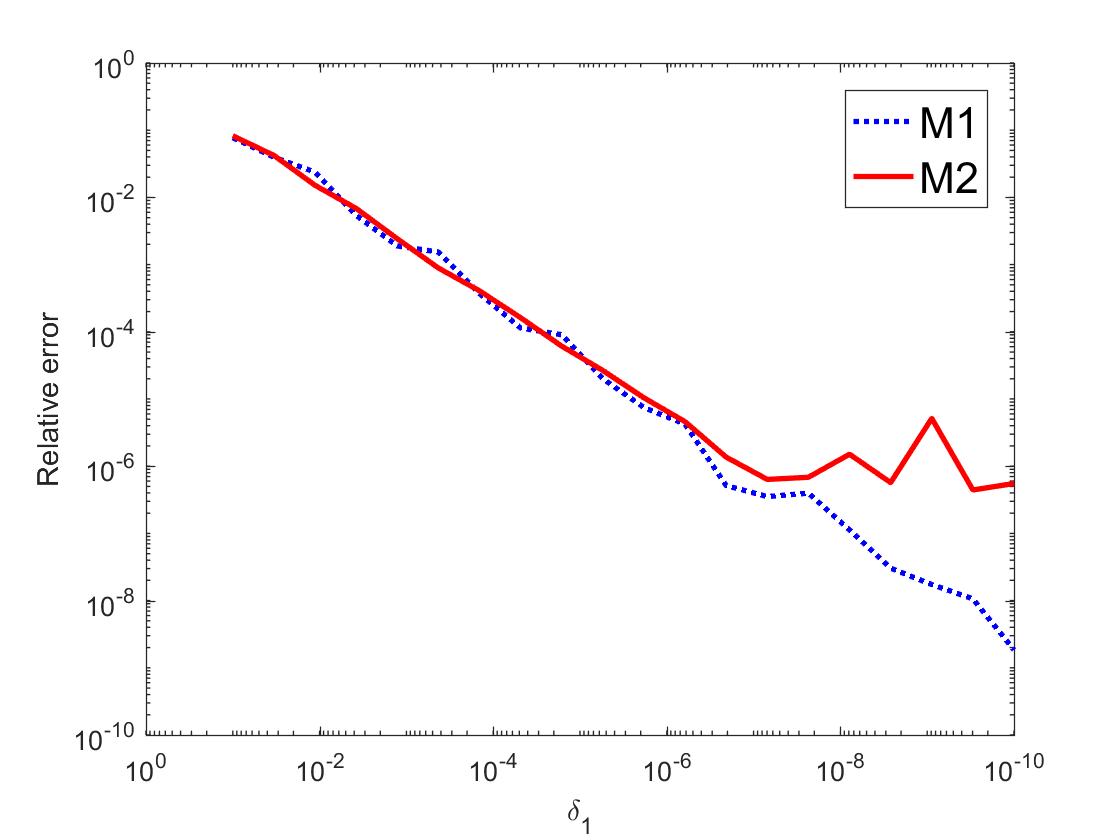}}
}%

\subfigure[\label{1c} 1st-order derivative ($\delta_1=1e-2$)] {
\resizebox*{6cm}{!}{\includegraphics{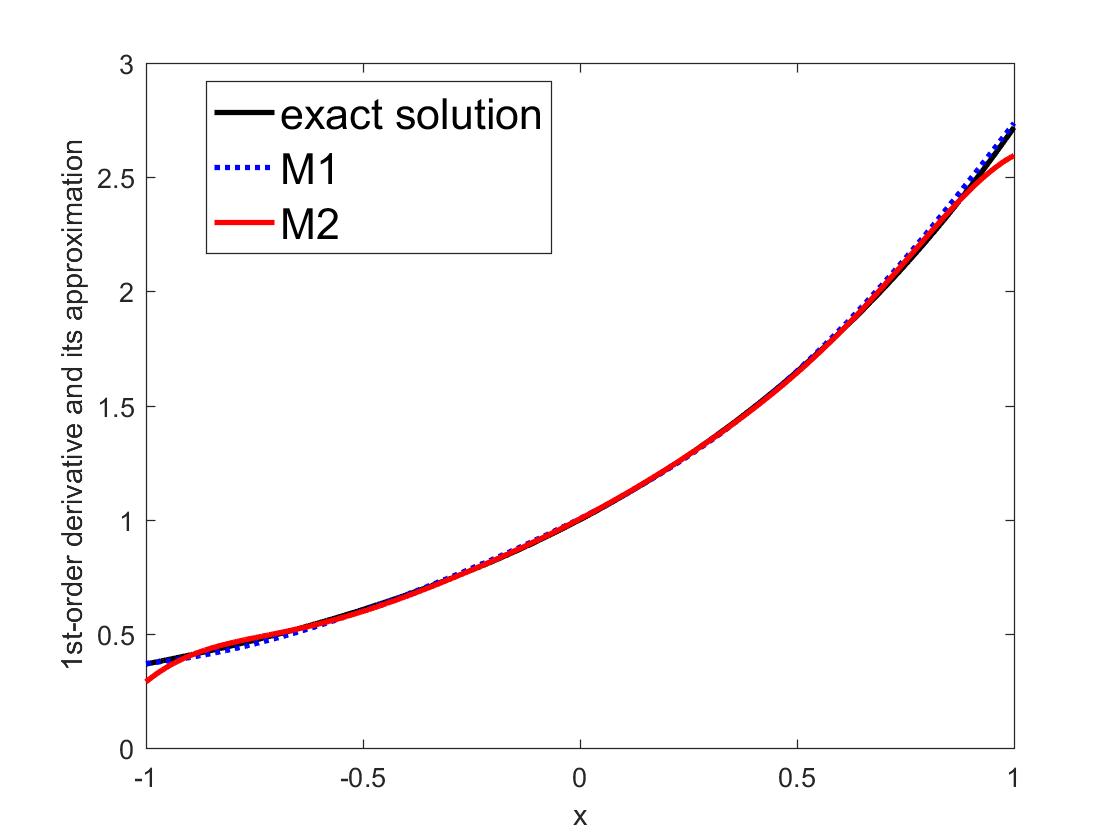}}
}%
\subfigure[\label{1d} Point-by-point error and  division nodes ($\delta_1=1e-2$)] {
\resizebox*{6cm}{!}{\includegraphics{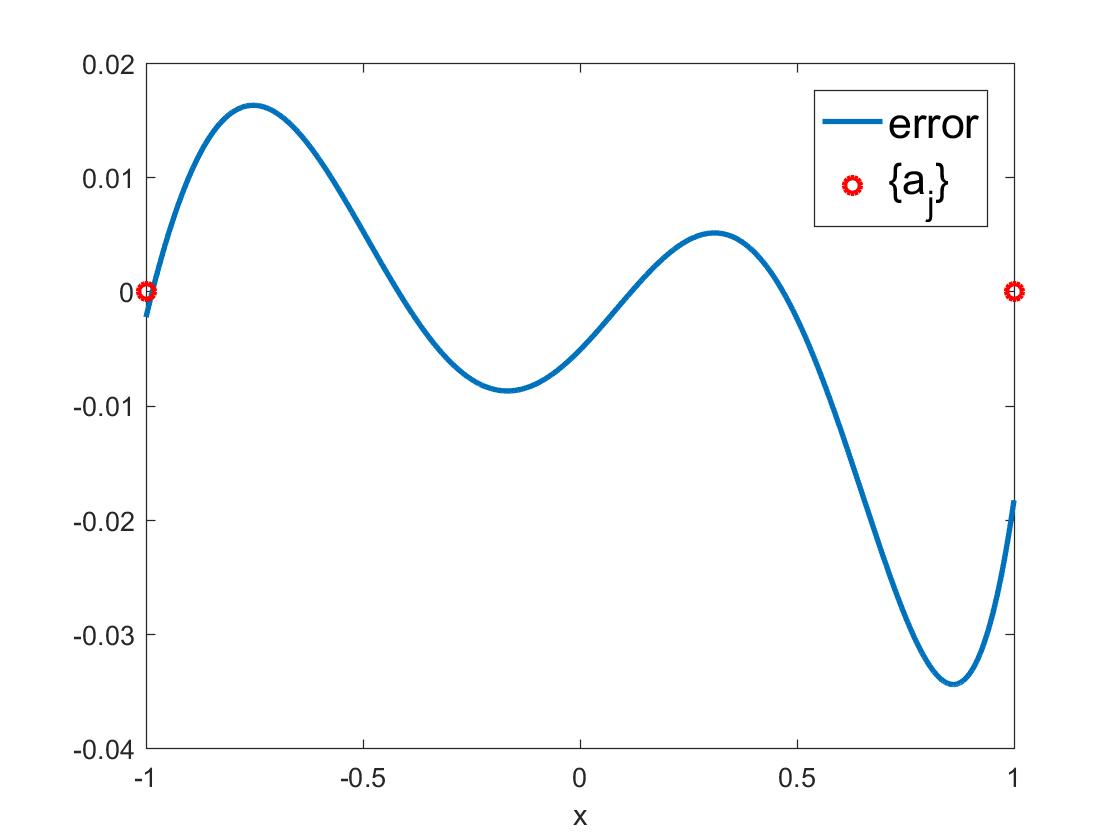}}
}%

{\caption{ Numerical results of $f_1(x)$}\label{Fig1}}
	\end{center}
\end{figure}

\begin{figure}
			\begin{center}
\subfigure[\label{2a} $f(x)=x^3-3x^2+x/2$] {
\resizebox*{6cm}{!}{\includegraphics{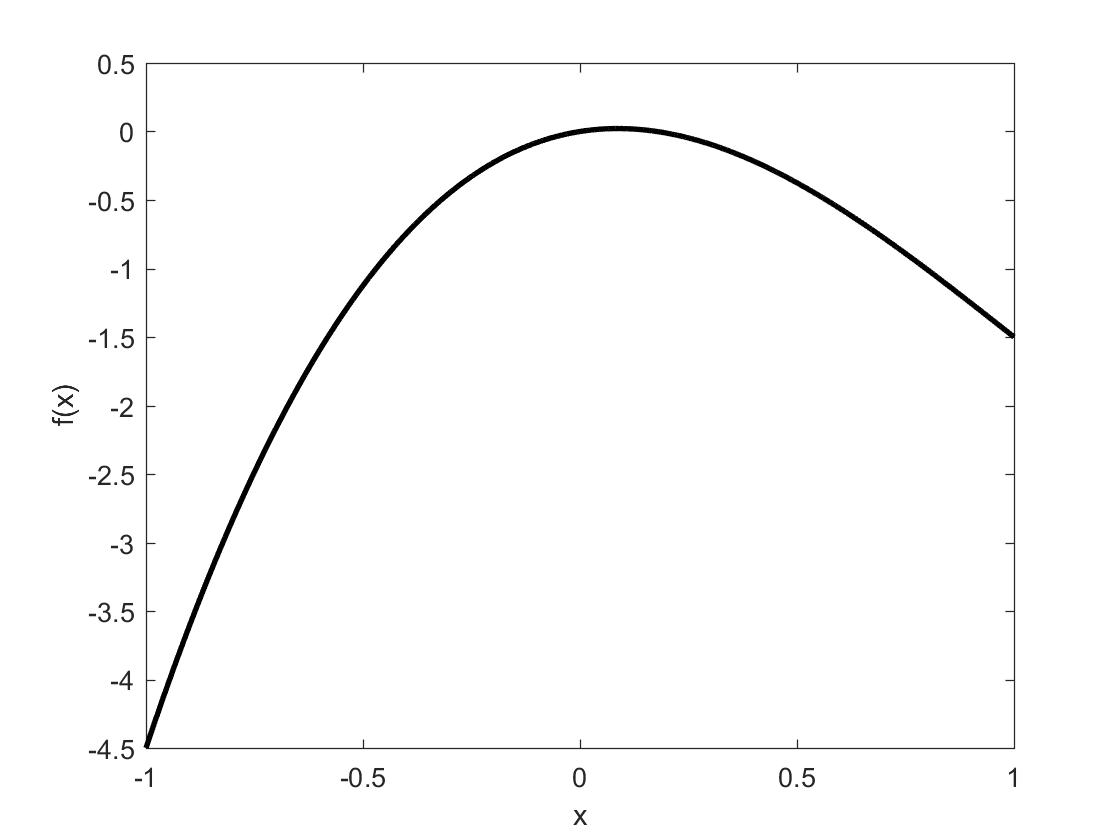}}
}%
\subfigure[\label{2b} relative error] {
\resizebox*{6cm}{!}{\includegraphics{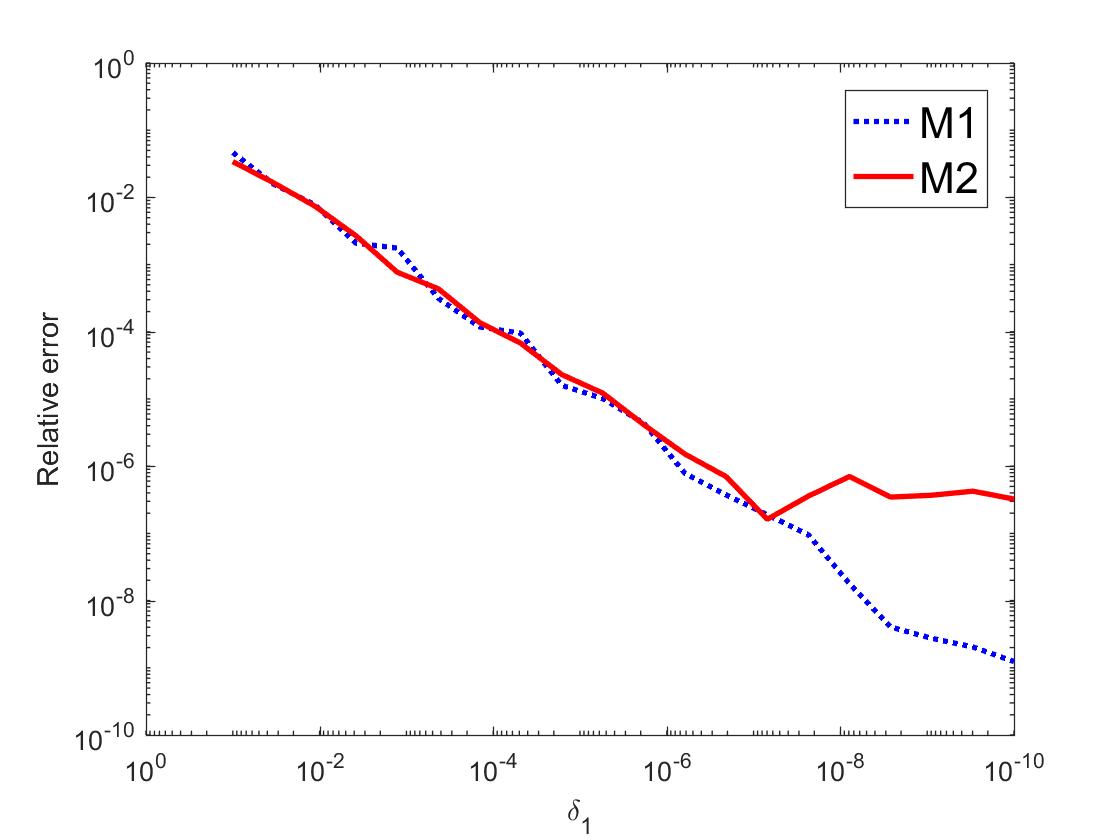}}
}%

\subfigure[\label{2c} 1st-order derivative ($\delta_1=1e-2$)] {
\resizebox*{6cm}{!}{\includegraphics{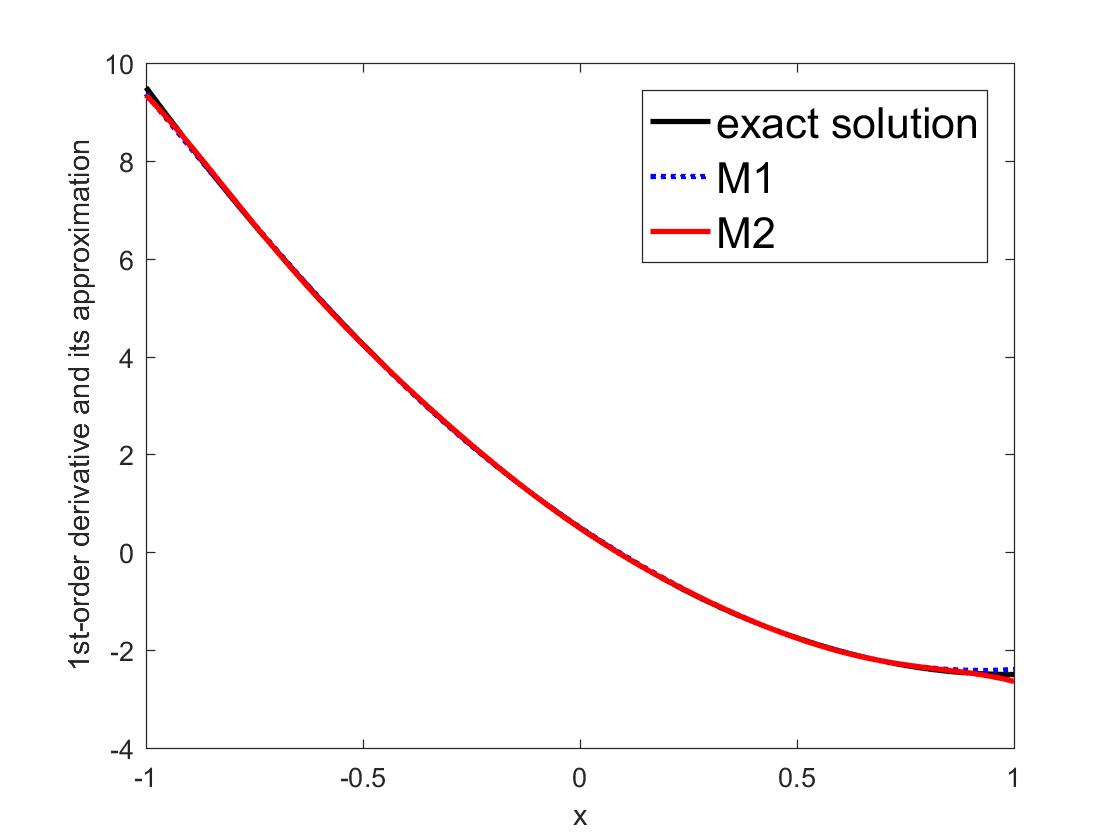}}
}%
\subfigure[\label{2d} Point-by-point error and  division nodes ($\delta_1=1e-2$)] {
\resizebox*{6cm}{!}{\includegraphics{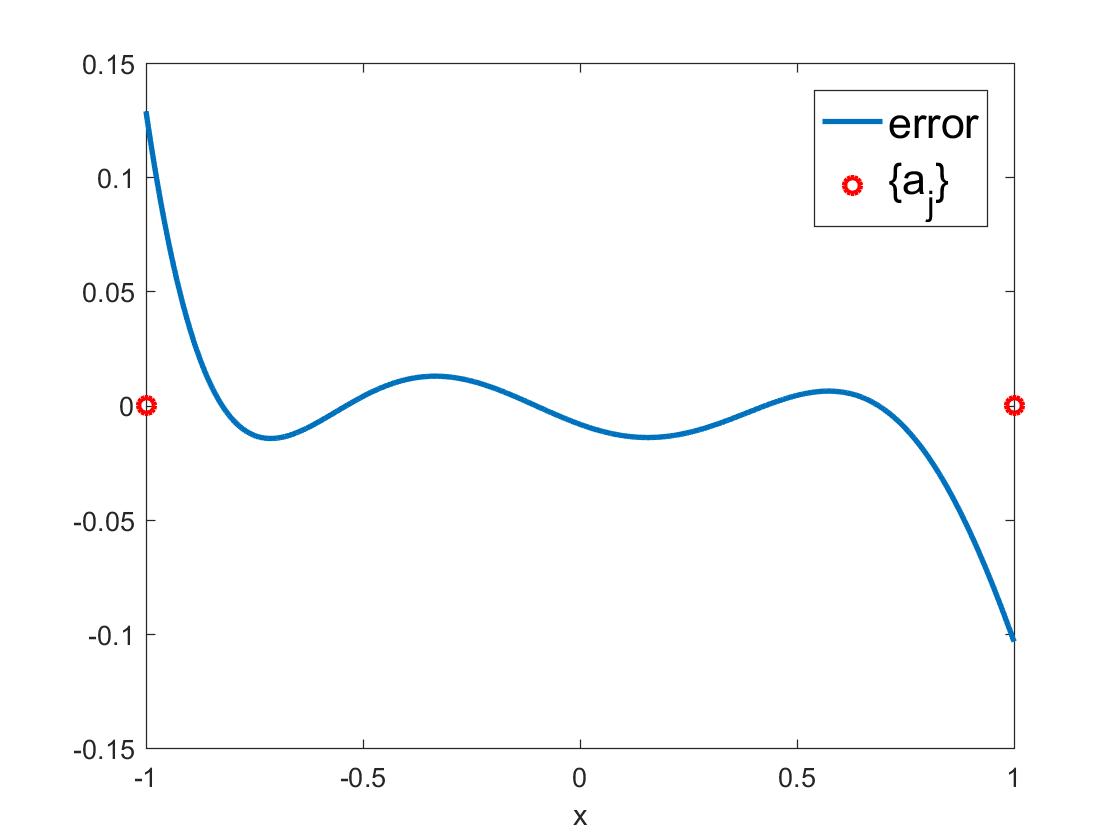}}
}%

{\caption{ Numerical results of $f_2(x)$}\label{Fig2}}
	\end{center}
\end{figure}

\begin{figure}
			\begin{center}
\subfigure[\label{3a} $f(x)=\cos\left(\frac{100}{1+25x^2}\right)$] {
\resizebox*{6cm}{!}{\includegraphics{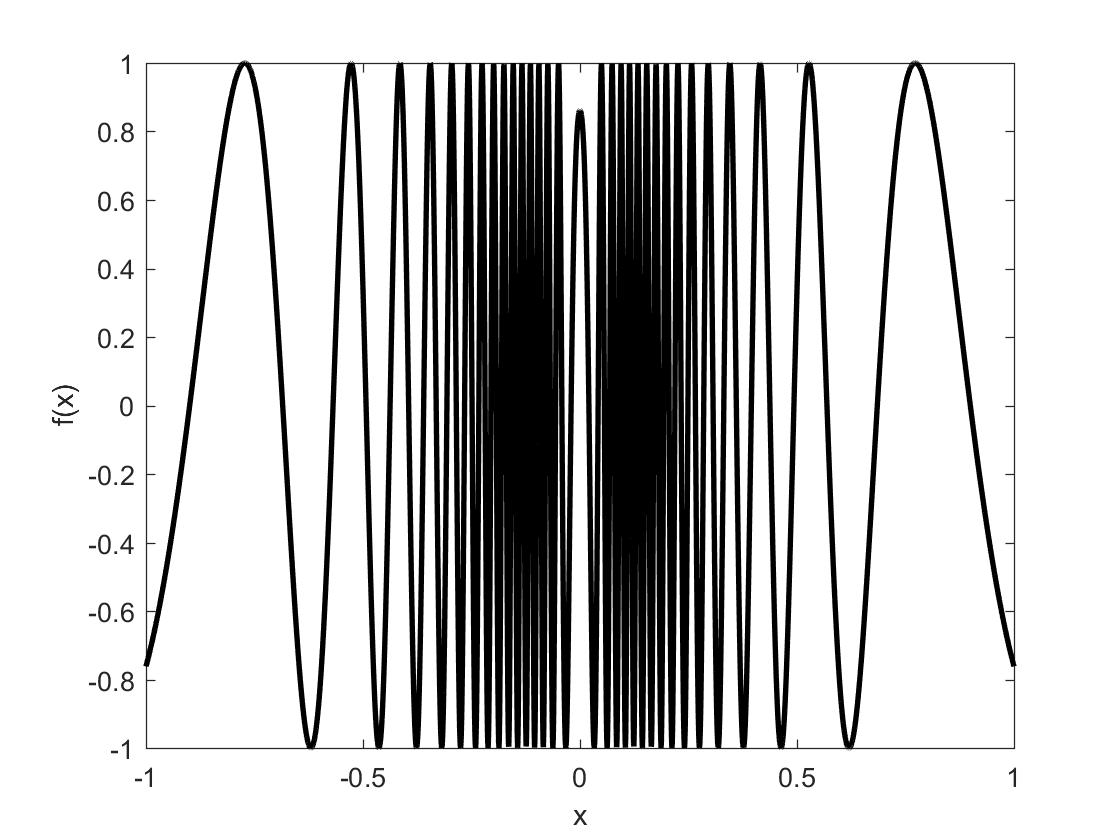}}
}%
\subfigure[\label{3b} relative error] {
\resizebox*{6cm}{!}{\includegraphics{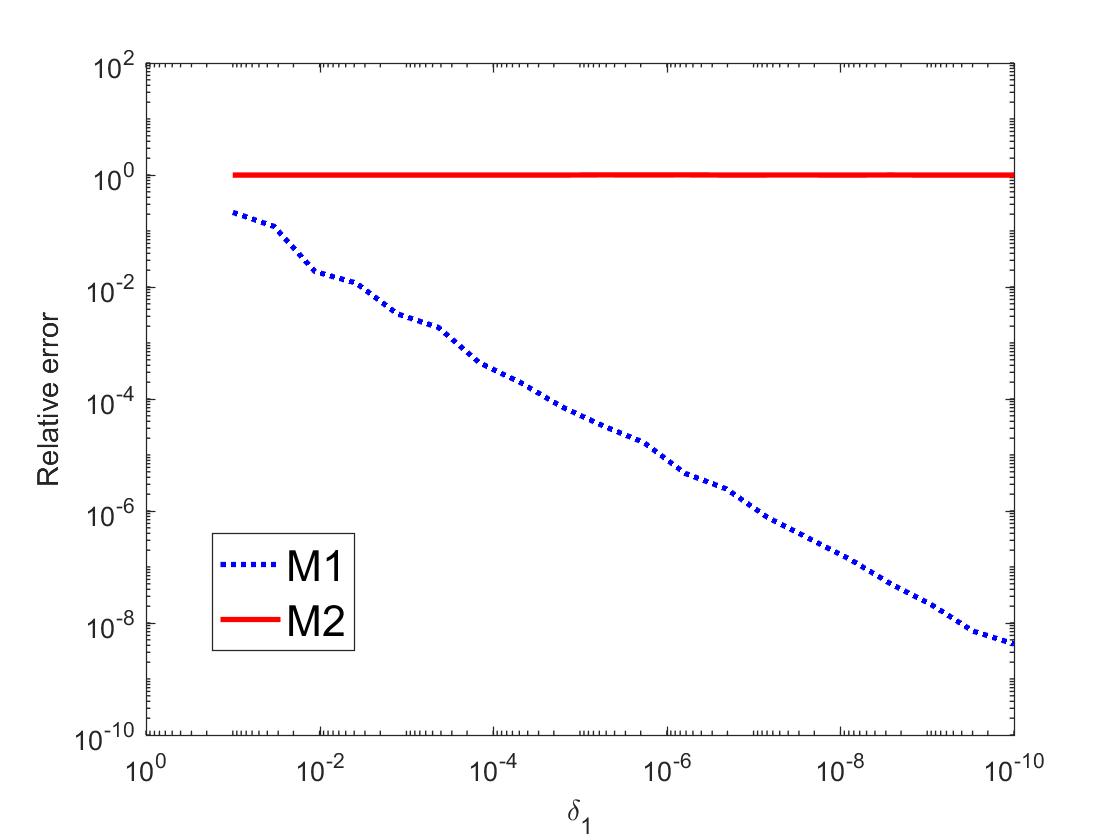}}
}%

\subfigure[\label{3c} 1st-order derivative ($\delta_1=1e-3$)] {
\resizebox*{6cm}{!}{\includegraphics{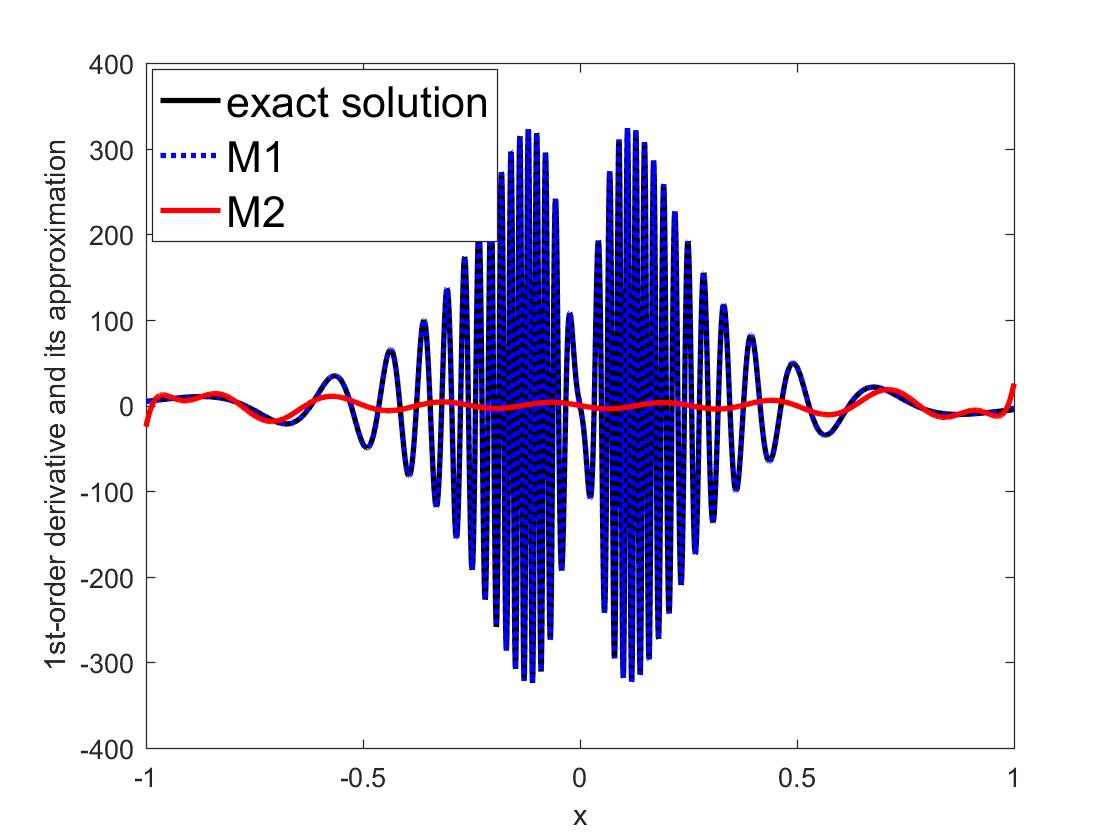}}
}%
\subfigure[\label{3d} Point-by-point error and  division nodes ($\delta_1=1e-3$)] {
\resizebox*{6cm}{!}{\includegraphics{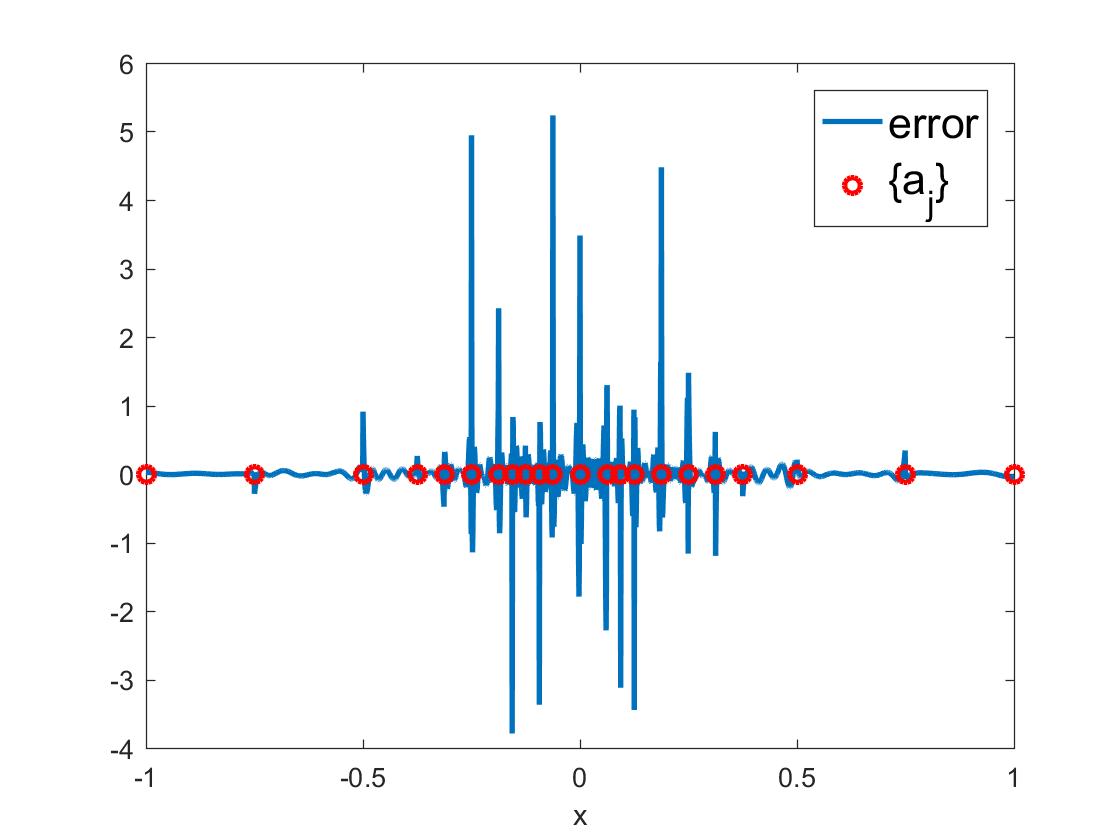}}
}%

{\caption{ Numerical results of $f_3(x)$}\label{Fig3}}
	\end{center}
\end{figure}

\begin{figure}
			\begin{center}
\subfigure[\label{4a} $f(x)=\text{erf}(x)$] {
\resizebox*{6cm}{!}{\includegraphics{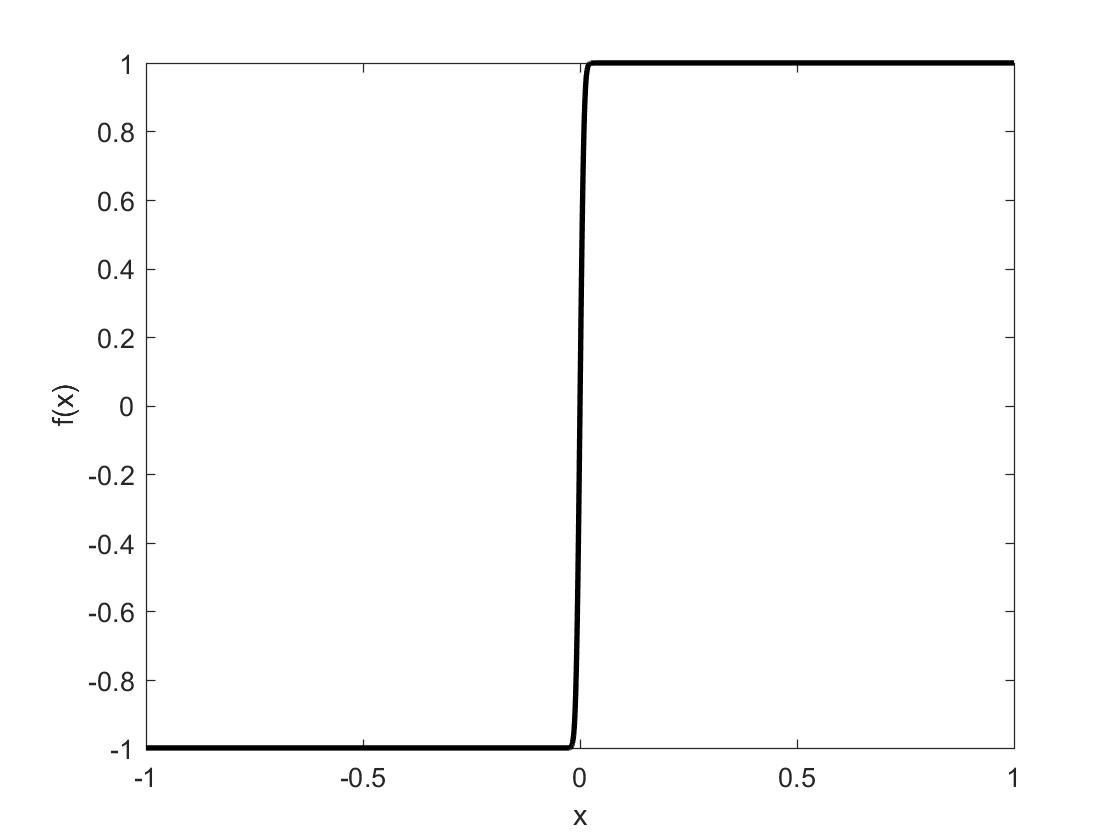}}
}%
\subfigure[\label{4b} relative error] {
\resizebox*{6cm}{!}{\includegraphics{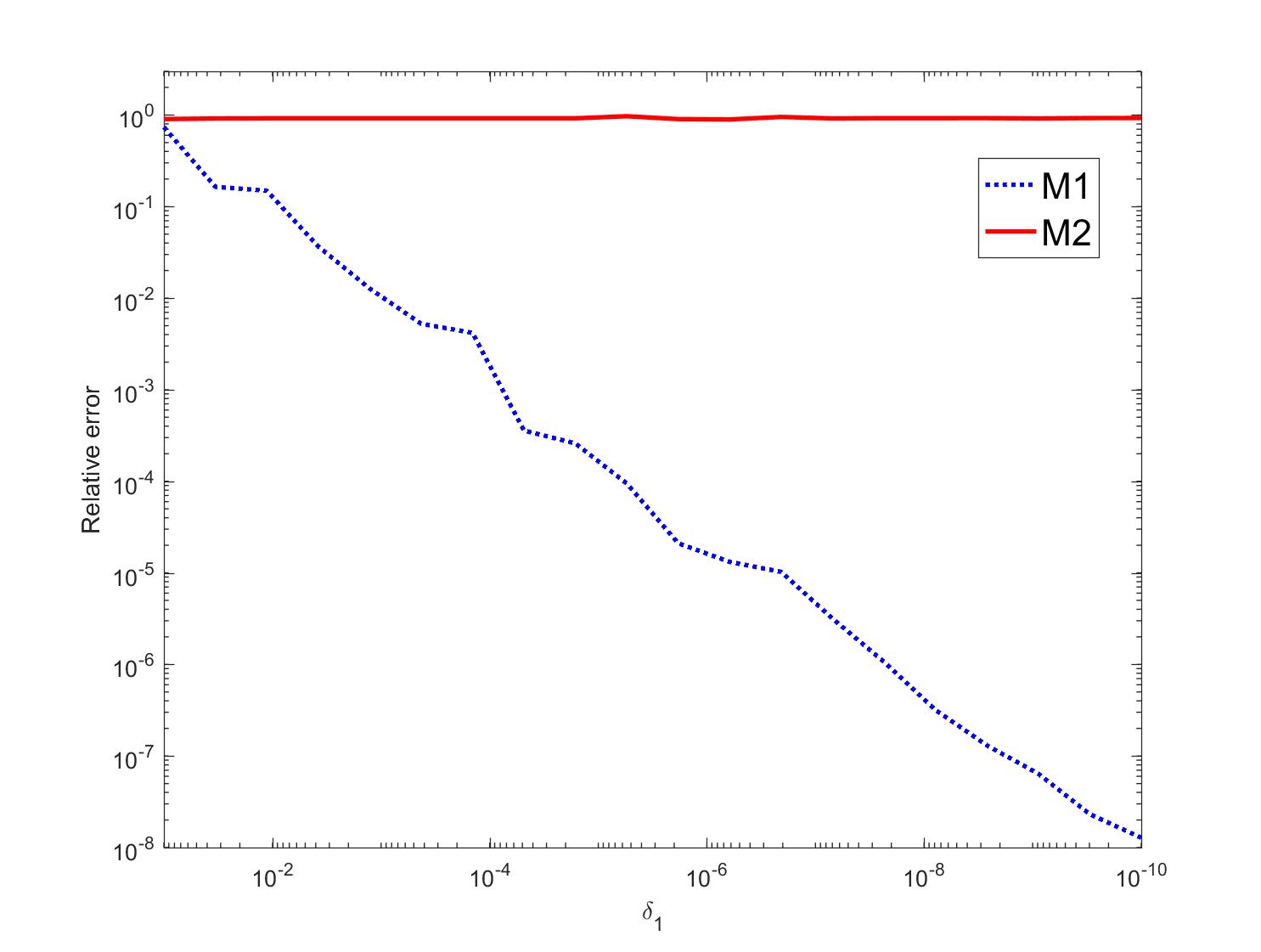}}
}%

\subfigure[\label{4c} 1st-order derivative ($\delta_1=1e-3$)] {
\resizebox*{6cm}{!}{\includegraphics{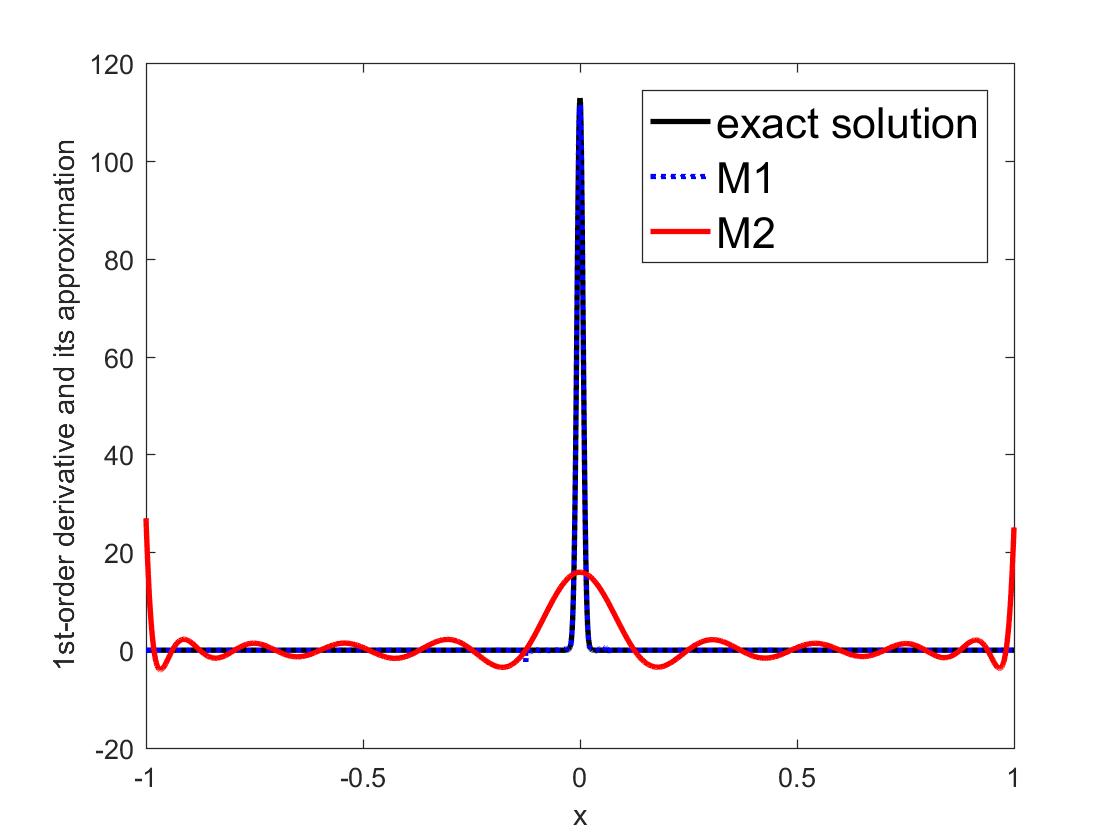}}
}%
\subfigure[\label{4d} Point-by-point error and  division nodes ($\delta_1=1e-3$)] {
\resizebox*{6cm}{!}{\includegraphics{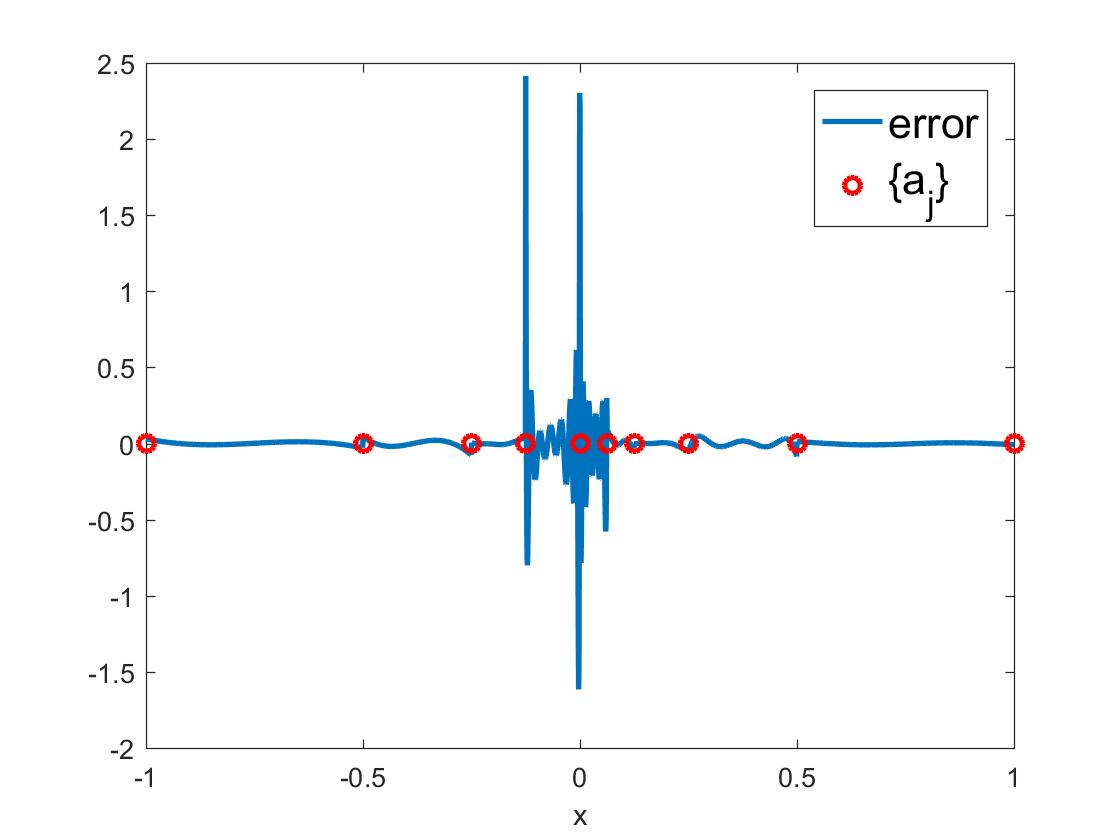}}
}%

{\caption{ Numerical results of $f_4(x)$}\label{Fig4}}
	\end{center}
\end{figure}

\begin{figure}
			\begin{center}
\subfigure[\label{5a} $f(x)=\cos(100x^2)$] {
\resizebox*{6cm}{!}{\includegraphics{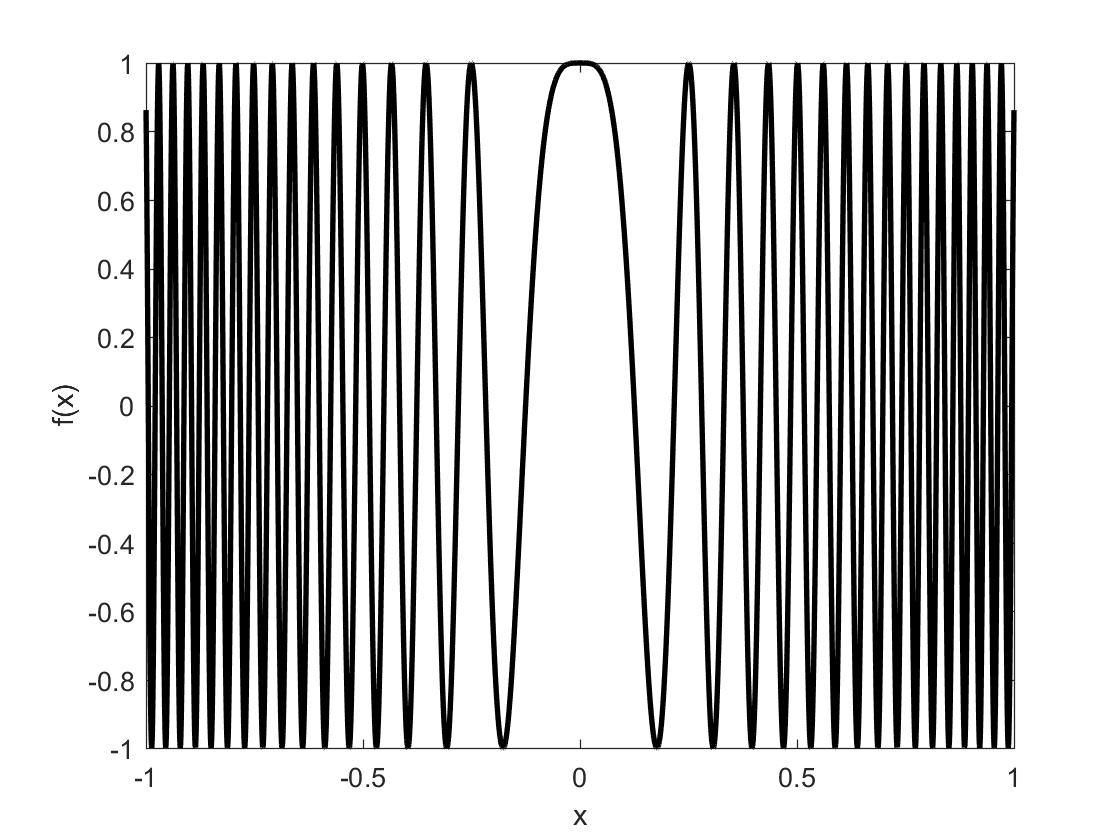}}
}%
\subfigure[\label{5b} relative error] {
\resizebox*{6cm}{!}{\includegraphics{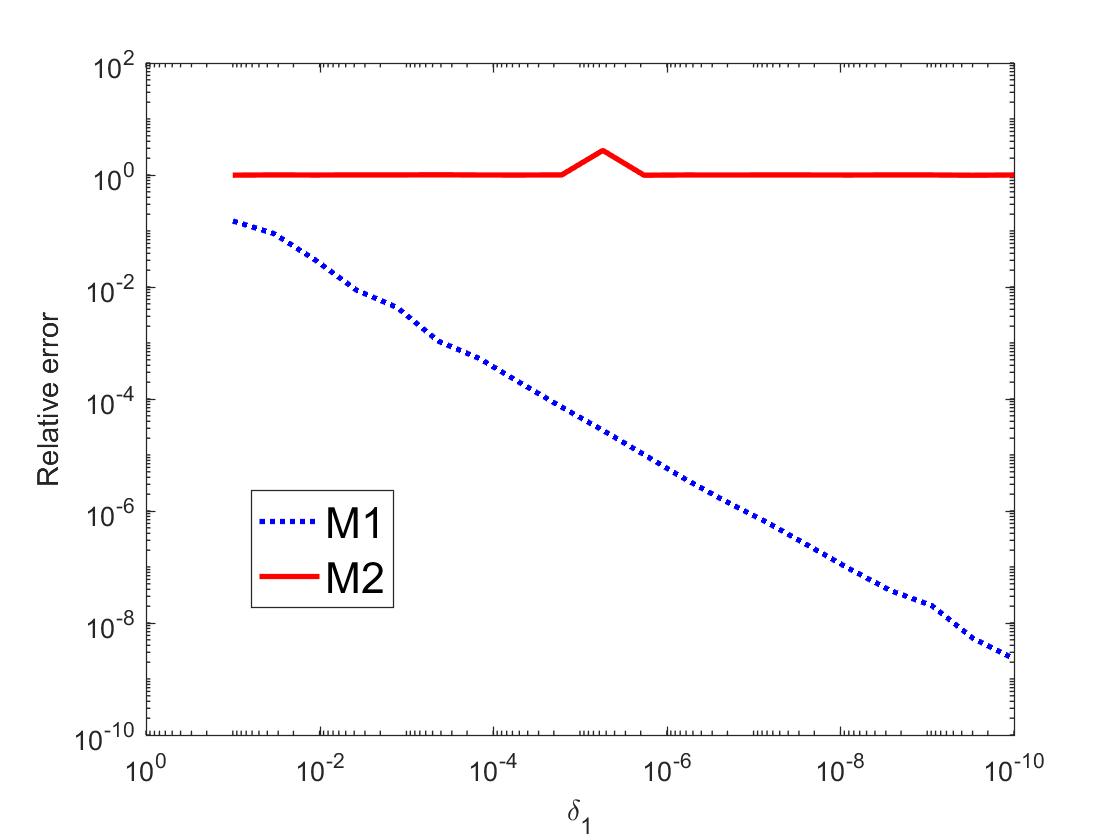}}
}%

\subfigure[\label{5c} 1st-order derivative ($\delta_1=1e-3$)] {
\resizebox*{6cm}{!}{\includegraphics{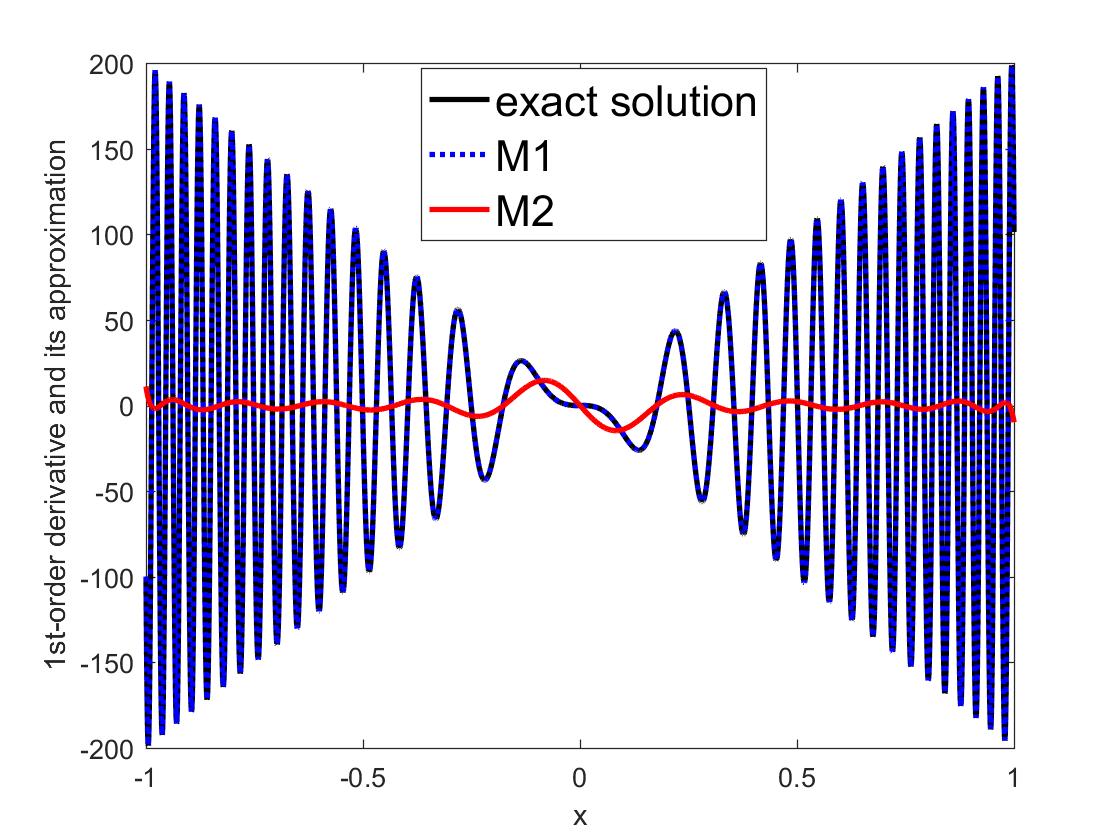}}
}%
\subfigure[\label{5d} Point-by-point error and  division nodes ($\delta_1=1e-3$)] {
\resizebox*{6cm}{!}{\includegraphics{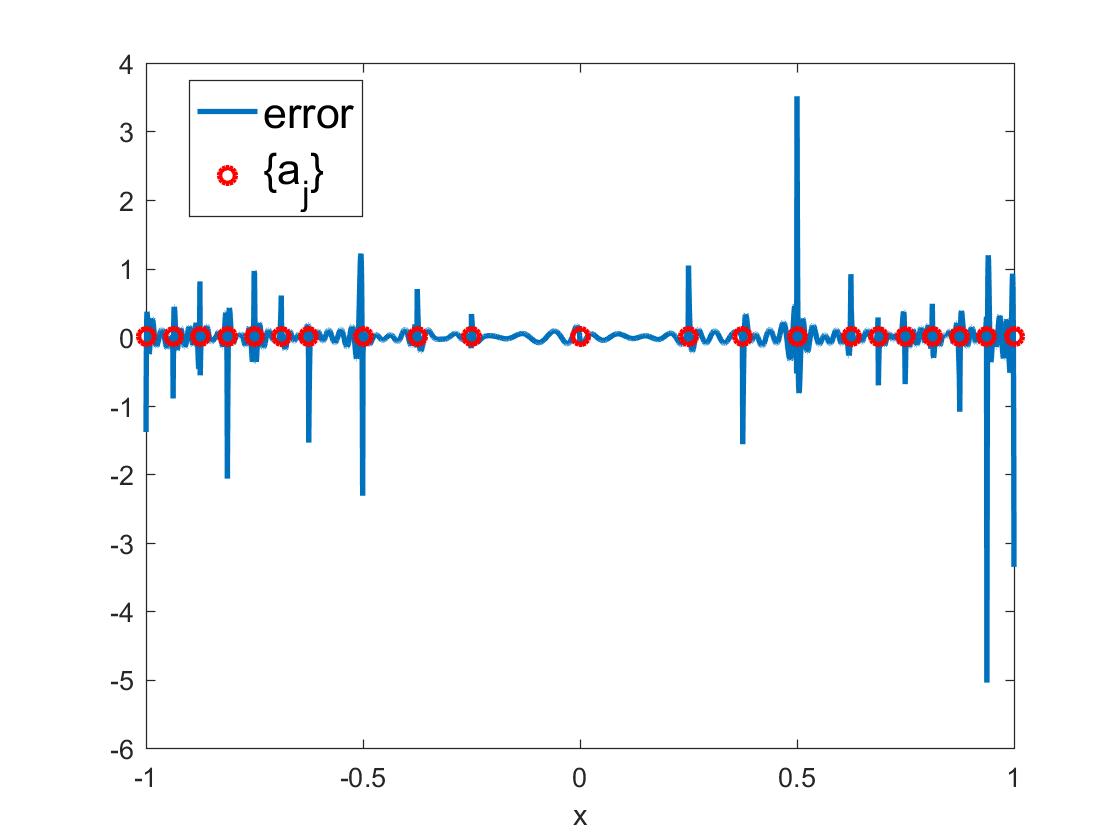}}
}%

{\caption{ Numerical results of $f_5(x)$}\label{Fig5}}
	\end{center}
\end{figure}

\begin{figure}
			\begin{center}
\subfigure[\label{6a} $f(x)=\frac{1}{1.1-x^2}$] {
\resizebox*{6cm}{!}{\includegraphics{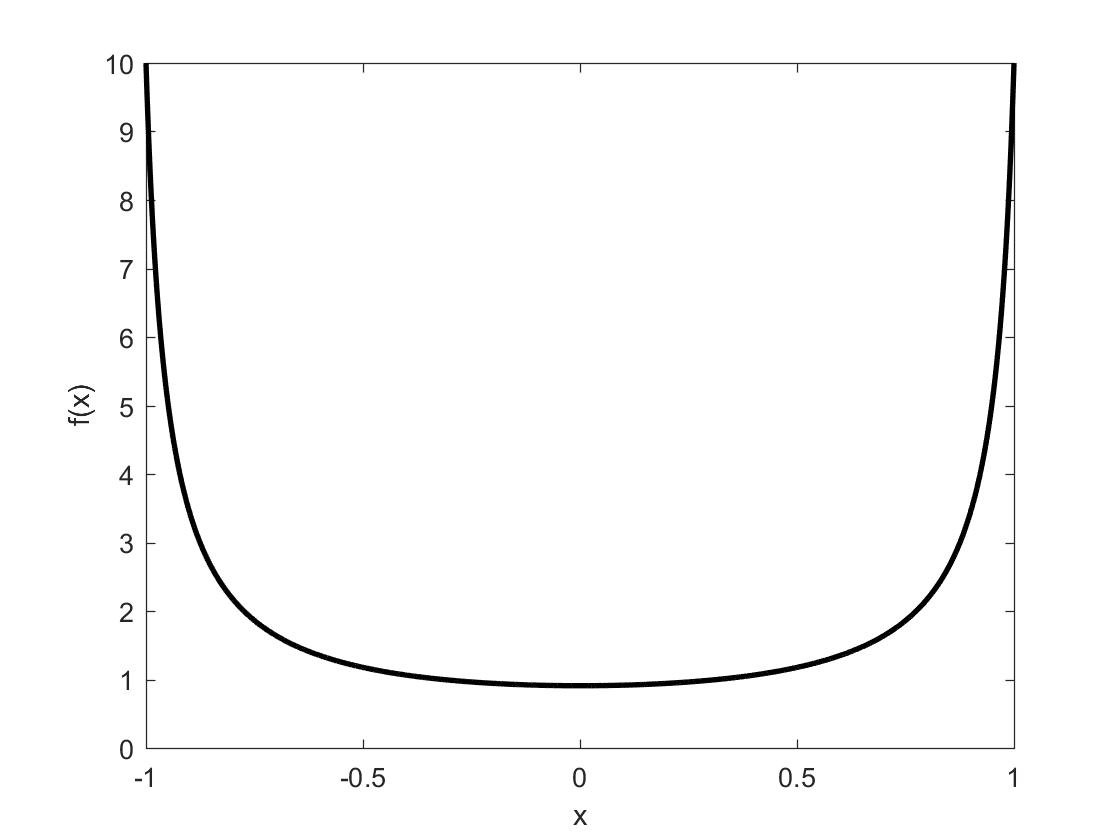}}
}%
\subfigure[\label{6b} relative error] {
\resizebox*{6cm}{!}{\includegraphics{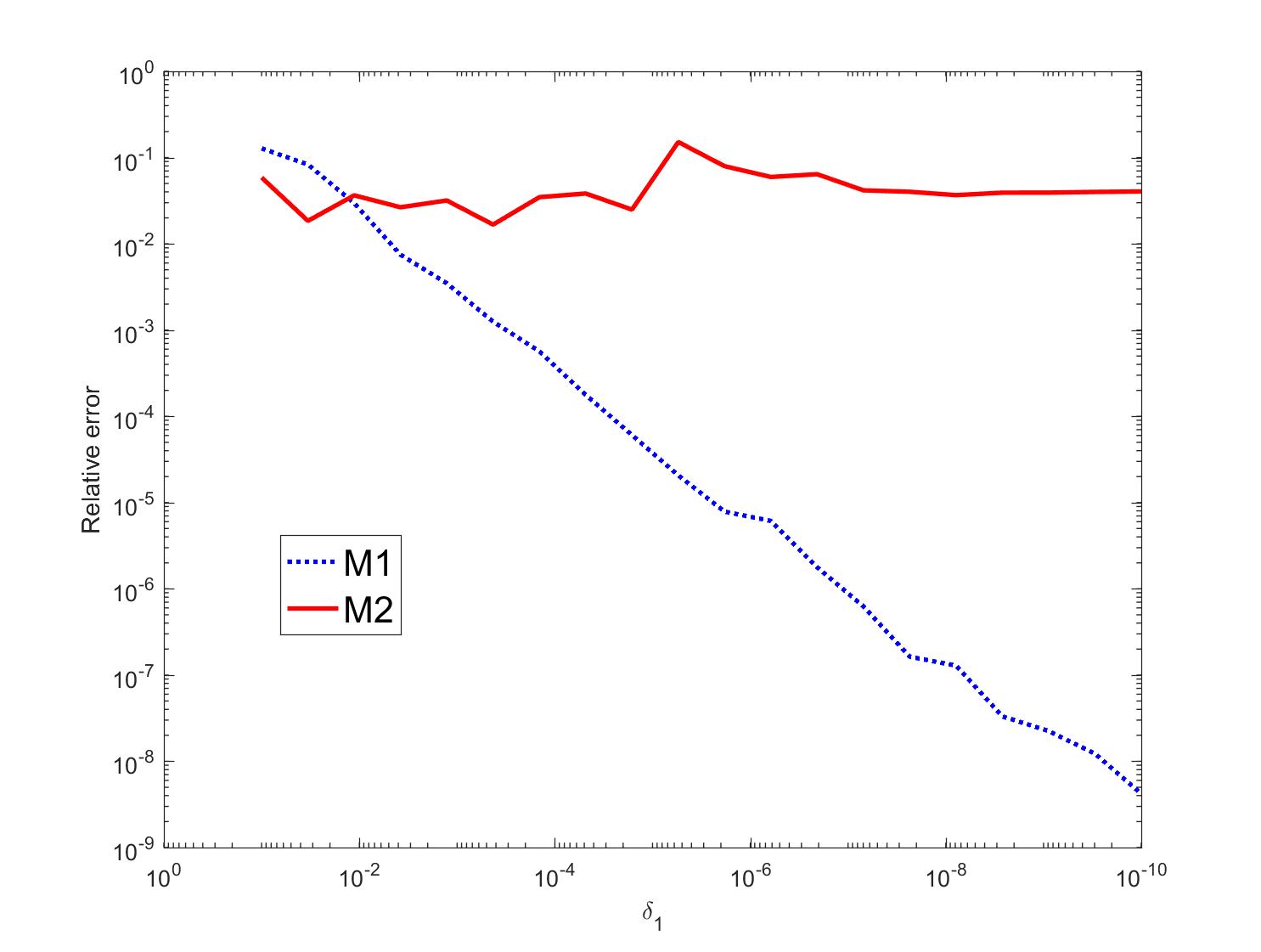}}
}%

\subfigure[\label{6c} 1st-order derivative ($\delta_1=1e-3$)] {
\resizebox*{6cm}{!}{\includegraphics{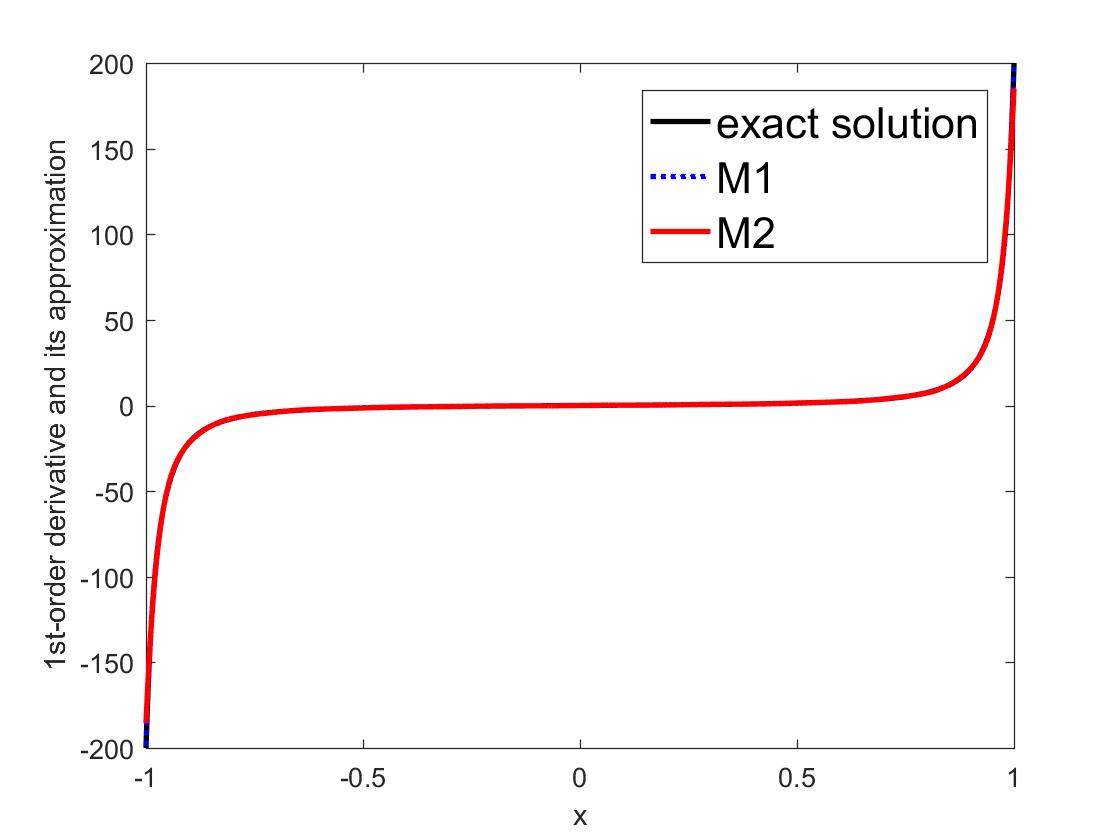}}
}%
\subfigure[\label{6d} Point-by-point error and  division nodes ($\delta_1=1e-3$)] {
\resizebox*{6cm}{!}{\includegraphics{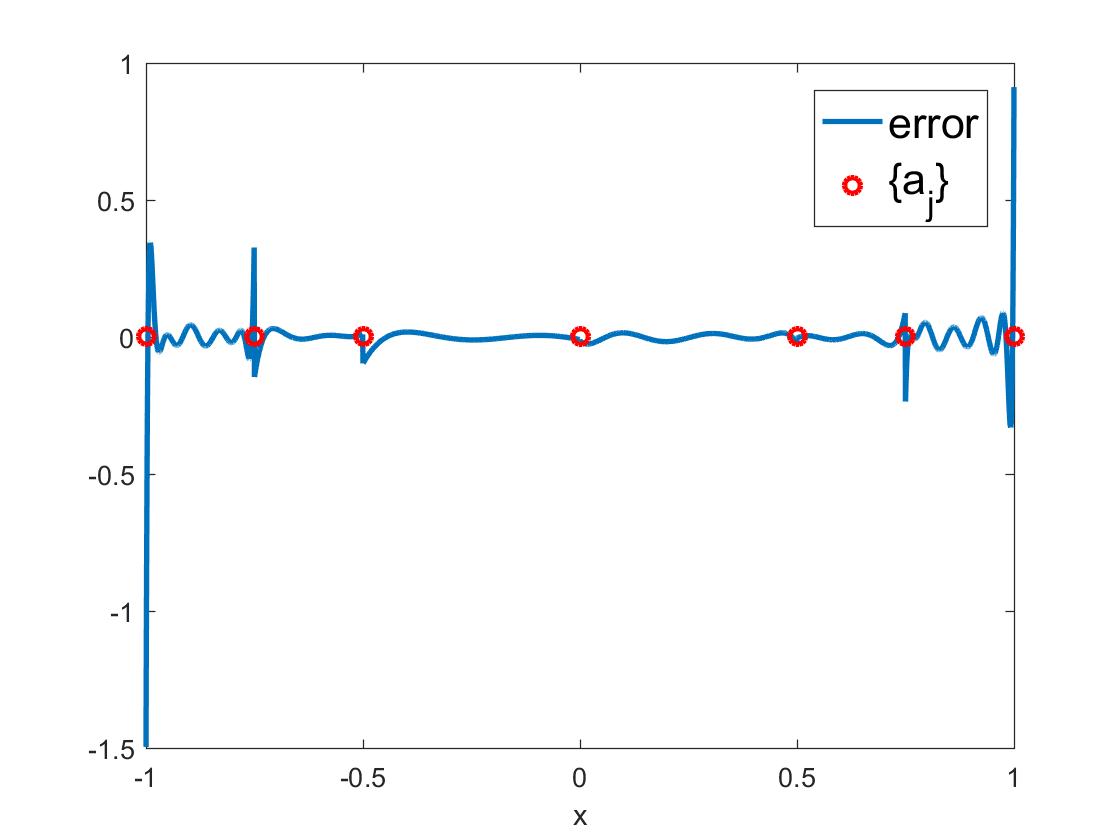}}
}%

{\caption{ Numerical results of $f_6(x)$}\label{Fig6}}
	\end{center}
\end{figure}

Figures~\ref{Fig1}--\ref{Fig6} present the numerical results.
For low-frequency functions (Figs.~\ref{1b} and~\ref{2b}), the two methods exhibit similar performance when the error level is relatively high.
However, once the error level drops below a certain threshold, the accuracy of M2 saturates and no longer improves with decreasing data error, whereas the approximation error of M1 continues to decrease.
This behavior arises because M2 solves the normal equations; see the analysis in~\cite{zhao2025new} for further details.

For high-frequency functions, whether the oscillations occur within the interval or are concentrated near the boundaries (Figs.~\ref{3b}--\ref{6b}), M2 fails to produce accurate derivative estimates, while M1 consistently delivers high-accuracy approximations in both cases.
Figures~\ref{1d}--\ref{6d} depict the distribution of the points $\{a_j\}$ for different test functions.
The distribution is largely consistent with the frequency characteristics of the underlying function, with slight deviations attributable to data noise.

Overall, these experiments demonstrate that the proposed method M1 effectively overcomes the limitations of M2, providing a stable and efficient approach for numerical differentiation.

\section{Conclusions and remarks\label{sec:conclusion}}
We conclude with the following remarks:
\begin{itemize}
  \item We have proposed a numerical differentiation method based on multi-interval Fourier extension.
  By partitioning the domain and reusing precomputed matrices, the method substantially reduces computational complexity.
  Moreover, the frequency-reduction effect of interval partitioning mitigates the machine-precision limitations inherent in high-order regularization methods, thereby enabling accurate reconstruction of high-frequency functions.
  This offers valuable insights for the treatment of ill-posed problems.

  \item Although numerical experiments for higher-order derivatives are not presented, Theorem~2 ensures that the proposed approach applies to derivatives of arbitrary order.

  \item The reconstructed function produced by our method is, in general, discontinuous.
  This discontinuity can be removed, if desired, by applying a piecewise spline post-processing step to enforce smoothness.
  As this is a straightforward procedure, it is omitted here.
\end{itemize}

\end{document}